\documentclass[a4paper,10pt]{article}
\usepackage[margin=3cm]{geometry}

\usepackage[T1]{fontenc}
\usepackage[utf8]{inputenc}

\usepackage[all,cmtip]{xy}

\usepackage{amsmath, mathtools}
\usepackage{amssymb}
\usepackage{amscd}
\usepackage{amsthm}
\usepackage{array}
\usepackage{makecell}
\usepackage{graphicx}
\usepackage{enumitem}
\usepackage{algpseudocode}
\usepackage{algorithm}
\usepackage{color}
\usepackage{url}

\theoremstyle{plain}
\newtheorem{theorem}{Theorem}
\newtheorem{lemma}[theorem]{Lemma}
\newtheorem{proposition}[theorem]{Proposition}

\newcommand{\N}{\mathbb{N}}

\newcommand{\R}{\mathbb{R}}

\newcommand{\forest}{\mathcal{W}}

\newcommand{\join}{\ast}

\newcommand{\ostar}{\mathrm{St}}
\newcommand{\cstar}{\overline{\mathrm{St}}}
\newcommand{\activestar}{\mathrm{Act\cstar}}
\newcommand{\link}{\mathrm{Lk}}

\newcommand{\field}{\mathbb{F}}
\newcommand{\complex}{\mathbb{K}}
\newcommand{\othercomplex}{\mathbb{L}}
\newcommand{\subcomplex}{\mathbb{L}}

\newcommand{\vertexset}{\mathcal{V}}
\newcommand{\othervertexset}{\mathcal{W}}

\newcommand{\nodeset}{N}
\newcommand{\node}{x}
\newcommand{\nodechildone}{y_1}
\newcommand{\nodechildtwo}{y_2}

\newcommand{\boundarymatrix}{\partial}
\newcommand{\rmatrix}{R}
\newcommand{\otherrmatrix}{S}
\newcommand{\arbmatrix}{A}

\newcommand{\dict}{D}
\newcommand{\innerdict}{CoF}

\newcommand{\dsmatrix}{M}

\newcommand{\chunk}{C}

\newcommand{\id}{\mathrm{id}}

\newcommand{\vectorspace}{\mathbb{V}}

\newcommand{\linearmap}{f}

\newcommand{\smap}{\phi}
\newcommand{\othersmap}{\psi}

\newcommand{\contr}[3]{(#1, #2) \leadsto #1}

\newcommand{\tower}{\mathcal{T}}
\newcommand{\filtration}{\mathcal{F}}

\newcommand{\inc}{\mathrm{inc}}

\newcommand{\towerdim}{\Delta}
\newcommand{\towerwidth}{\omega}

\setlist{itemsep=0mm}

\title{Barcodes of Towers and a Streaming Algorithm for Persistent Homology}
\author{Michael Kerber \thanks{Graz University of Technology, Austria} \and Hannah Schreiber \footnotemark[1]}
\date{\today} 

\begin{document}

\maketitle
\begin{abstract}
	A tower is a sequence of simplicial complexes connected by simplicial maps.
	We show how to compute a filtration, a sequence of nested simplicial complexes,
	with the same persistent barcode as the tower.
	Our approach is based on the coning strategy by Dey et al. (SoCG 2014).
	We show that a variant of this approach yields a filtration that is asymptotically only
	marginally larger than the tower and can be efficiently computed by a streaming
	algorithm, both in theory and in practice.
	Furthermore, we show that our approach can be combined with a streaming algorithm
	to compute the barcode of the tower via matrix reduction. The space complexity
	of the algorithm does not depend on the length of the tower, but the maximal size
	of any subcomplex within the tower. Experimental evaluations show that our approach
	can efficiently handle towers with billions of complexes.
\end{abstract}

\section{Introduction}

\paragraph{Motivation and problem statement.}
Persistent homology~\cite{elz-topological,carlsson-survey,edelsbrunnerharer} is a paradigm
to analyze how topological properties of general data sets evolve across
multiple scales. Thanks to the success of the theory in finding applications
(see, e.g., \cite{oudot-book,kerber-survey} for recent enumerations), there is a growing
demand for efficient computations of the involved topological invariants.

In this paper, we consider a sequence of simplicial complexes $(\complex_i)_{i=0,\ldots,m}$
and simplicial maps $\smap_i:\complex_i\rightarrow \complex_{i+1}$ connecting them, 
calling this data a \emph{(simplicial) tower} of length $m$.
Applying the homology functor with an arbitrary field, 
we obtain a \emph{persistence module}, 
a sequence of vector spaces connected by linear maps. Such a module decomposes into a \emph{barcode},
a collection of intervals, each representing a homological feature in the tower
that spans over the specified range of scales. 

Our computational problem is to compute the barcode of a given tower efficiently.
The most prominent case of a tower is when all maps $f_i$ are inclusion maps.
In this case one obtains a \emph{filtration}, a sequence of nested simplicial complexes.
A considerable amount of work went into the study of fast algorithms for the filtration case,
which culminated in fast software libraries for this task. 
The more general case of towers recently received
growing interest in the context of sparsification technique for the Vietoris-Rips and \v{C}ech
complexes; see the related work section below for a detailed discussion.

\paragraph{Results.}
As our first result, we show that any tower can be \emph{efficiently} converted into
a \emph{small} filtration with the same barcode. Using the well-known concept
of \emph{mapping cylinders} from algebraic topology~\cite{hatcher}, it is easy to see
that such a conversion is possible in principle. 
Dey, Fan, and Wang~\cite{dfw-computing} give an explicit construction, called ``coning'',
for the generalized case of \emph{zigzag towers}. Using a simple
variant of their coning strategy, we obtain a filtration whose size is only
marginally larger than the length of the tower in the worst case. 
Furthermore, we experimentally show that the size is even smaller on realistic instances.

To describe our improved coning strategy, 
we discuss the case that a simplicial map in the tower contracts two vertices $u$ and $v$.
The coning strategy by Dey et al.\ proposes
to join $u$ with the closed star of $v$, making all incident simplices of $v$ incident
to $u$ without changing the homotopy type. 
The vertex $u$ is then taken as the representative
of the contracted pair in the further processing of the tower.
We refer to the number of simplices that the join operation adds to the complex
as the \emph{cost} of the contraction.
Quite obviously, the method is symmetric in $u$ and $v$, and we have two choices
to pick the representative, leading to potentially quite different costs. 
We employ the self-evident strategy to pick the representative that leads to smaller costs
(somewhat reminiscent of the ``union-by-weight'' rule
in the union-find data structure~\cite[\S 21]{cormen}).
Perhaps surprisingly, this idea leads to an asymptotically
improved size bound on the filtration. We prove this by an abstraction
to path decompositions on weighted forest which might be of some
independent interest. Altogether, the worst-case size of the filtration is
$O(\towerdim\cdot n \cdot \log(n_0))$, where $\towerdim$ is the maximal dimension of any complex
in the tower, and $n$/$n_0$ is the number of simplices/vertices added to the tower.

We also provide a conversion algorithm whose time complexity is roughly proportional
to the total number of simplices in the resulting filtration.
One immediate benefit is a generic solution to compute barcodes of towers:
just convert the tower to a filtration and apply one of the efficient implementations
for barcodes of filtrations. Indeed, we experimentally show that on not-too-large
towers, our approach is competitive with, and sometimes outperforms \textsc{Simpers},
an alternative approach that computes the barcode of towers with \emph{annotations},
a variant of the persistent cohomology algorithm. 

Our second contribution is a space-efficient version of the just mentioned algorithmic pipeline
that is applicable to very large towers.
To motivate the result, let the \emph{width} 
of a tower denote the maximal size of any simplicial complex among the $\complex_i$.
Consider a tower with a very large length (say, larger than the number of bytes in main memory) 
whose width remains relatively small.
In this case, our conversion algorithm yields a filtration that is very large as well.
Most existing implementations for barcode computation 
read the entire filtration on initialization and must be 
converted to streaming algorithm to handle such instances.
Moreover, algorithms based on matrix reduction are required to keep previously reduced columns
because they might be needed in subsequent reduction steps.
This leads to a high memory consumption for the barcode computation.

We show that with minor modifications, the standard persistent algorithm
can be turned into a streaming algorithm with smaller space complexity
in the case of towers.
The idea is that upon contractions, simplices become \emph{inactive} and cannot get additional
cofaces. Our approach makes use of this observation by modifying the boundary matrix
such that columns associated to inactive simplices can be removed. 
Combined with our conversion algorithm, we can compute the barcode of a tower of width $\towerwidth$
keeping only up to $O(\towerwidth)$ columns of the boundary matrix in memory.
This yields a space complexity of $O(\towerwidth^2)$ and a time complexity
of $O((\towerdim\cdot n \cdot \log(n_0))\towerwidth^2)$ in the worst case.
We implemented a practically improved variant
that makes use of additional heuristics to speed up the barcode
computation in practice and resembles the \emph{chunk algorithm}
presented in~\cite{chunk}.

We tested our implementation on various challenging data sets. The source code of the implementation 
is available\footnote{at \url{https://bitbucket.org/schreiberh/sophia/}} and the software was named Sophia.

\paragraph{Related work.}
Already the first works on persistent homology pointed out the existence of efficient
algorithm to compute the barcode invariant (or equivalently, the persistent diagram)
for filtrations~\cite{elz-topological,zc-computing}. 
As a variant of Gaussian elimination, the worst-case complexity is cubic.
Remarkable theoretical follow-up results are a persistence algorithm in matrix multiplication
time~\cite{mms-zigzag}, an output-sensitive algorithm 
to compute only high-persistent features with linear space complexity~\cite{ck-output}, 
and a conditional lower bound on the complexity relating the problem 
to rank computations of sparse matrices~\cite{ep-computational}.

On realistic instances, the standard algorithm has shown a quasi-linear behavior in practice
despite its pessimistic worst-case complexity. Nevertheless, many improvements of the standard algorithm
have been presented in the last years which improve the runtime by several orders
of magnitude. One line of research exploits the special structure of the boundary matrix
to speed up the reduction process~\cite{ck-twist}. This idea has led to efficient
parallel algorithms for persistence in shared~\cite{chunk} and distributed memory~\cite{bkr-distributed}.
Moreover, of same importance as the reduction strategy is an appropriate choice of
data structures in the reduction process as demonstrated by the \textsc{PHAT} library~\cite{bkrw-phat}.
A parallel development was the development of dual algorithms using persistent cohomology,
based on the observation that the resulting barcode is identical~\cite{smv-duality}.
The \emph{annotation algorithm}~\cite{dfw-computing,bdm-compressed} is an optimized variant 
of this idea realized in the \textsc{Gudhi} library~\cite{mbgy-gudhi}.
It is commonly considered as an advantage of annotations 
that only a cohomology basis must be kept during the reduction process,
making it more space efficient than reduction-based approaches.
We refer to the comparative study~\cite{optgh-roadmap}
for further approaches and software for persistence on filtrations.

Moreover, generalizations of the persistence paradigm are an active field of study. 
\emph{Zigzag persistence} is a variant of persistence where the maps in the filtration are allowed
to map in either direction (that is, either 
$\smap_i:\complex_i\hookrightarrow \complex_{i+1}$ or $\smap_i:\complex_i\hookleftarrow \complex_{i+1}$).
The barcode of zigzag filtrations is well-defined~\cite{zigzag} as a consequence of Gabriel's theorem~\cite{gabriel}
on decomposable quivers~-- see~\cite{oudot-book} for a comprehensive introduction.
The initial algorithms to compute this barcode~\cite{csm-zigzag} has been improved recently~\cite{mo-zigzag}.
Our case of towers of complexes and simplicial maps can be modeled as a zigzag filtration and therefore
sits in-between the standard and the zigzag filtration case.

Dey et al.~\cite{dfw-computing} described the first efficient algorithm to compute the barcode of towers.
Instead of the aforementioned coning approach explained in their paper, their implementation handles contractions with an empirically smaller
number of insertions, based on the link condition. 
Recently, the authors have released the \textsc{SimPers} library%
\footnote{\url{http://web.cse.ohio-state.edu/~tamaldey/SimPers/Simpers.html}}
that implements their annotation algorithm from the paper.

The case of towers has received recent attention in the context of approximate 
Vietoris-Rips and \v{C}ech filtrations.
The motivation for approximation is that the (exact) topological analysis of a set of $n$ points in $d$-dimensions
requires a filtration of size $O(n^{d+1})$ which is prohibitive for most interesting input sizes.
Instead, one aims for a filtration or tower of much smaller size, with the guarantee that the approximate barcode
will be close to the exact barcode (``close'' usually means that the bottleneck distance between the barcodes
on the logarithmic scale is upper bounded; we refer to the cited works for details).
The first such type of result by Sheehy~\cite{sheehy-linear} resulted in a approximate filtration;
however, it has been observed that passing to towers allows more freedom in defining the approximation complexes
and somewhat simplifies the approximation schemes conceptually. 
See~\cite{dfw-computing,bs-approximating,ks-approximate,ckr-polynomial} for examples.
Very recently, the SimBa library~\cite{dsw-simba} brings these theoretical approximation techniques 
for Vietoris-Rips complexes into practice.
The approach consists of a geometric layer to compute a tower, and an algebraic layer to compute its barcode,
for which they use \textsc{SimPers}. Our approach can be seen as an alternative realization of this algebraic layer.

\medskip

This paper is a more complete version of the conference paper \cite{ks_socg17}. 
It provides missing proof details from \cite{ks_socg17} and a conclusion, both omitted in the former version for space restrictions.
In addition, the experimental results were redone with the most recent versions of the corresponding software libraries. 
Furthermore, the new subsection~\ref{ssec:tightness_bounds} discusses the tightness of the complexity bound of our first main result.

\paragraph{Outline.} 
We introduce the necessary basic concepts in Section~\ref{sec:background}.
We describe our conversion algorithm from general towers to barcodes
in Section~\ref{sec:from_towers_to_filtrations}.
The streaming algorithm for persistence is discussed 
in Section~\ref{sec:persistence_by_streaming}.

\section{Background}\label{sec:background}

\paragraph{Simplicial Complexes.}
Given a finite \emph{vertex set} $V$, 
a \emph{simplex} is merely a non-empty subset of $V$;
more precisely, a \emph{$k$-simplex} is a subset consisting
of $k+1$ vertices, and $k$ is called the \emph{dimension} of the simplex.
Throughout the paper, we will denote simplices by small Greek letters,
except for vertices ($0$-simplices) which we denote by $u$, $v$, and $w$.
For a $k$-simplex $\sigma$, a simplex $\tau$ is a \emph{face} of $\sigma$
if $\tau\subseteq\sigma$. If $\tau$ is of dimension $\ell$, we call it
a $\ell$-face. If $\ell<k$, we call $\tau$ a \emph{proper face} of $\sigma$,
and if $\ell=k-1$, we call it a \emph{facet}. 
For a simplex $\sigma$ and a vertex $v\notin\sigma$, we define the \emph{join}
$v\join\sigma$ as the simplex $\{v\}\cup \sigma$.
These definition are inspired by visualizing a $k$-simplex as the convex hull
of $k+1$ affinely independent points in $\R^k$, but we will not need
this geometric interpretation in our arguments.

An \emph{(abstract) simplicial complex} $\complex$ over $V$ is a set of simplices
that is closed under taking faces. 
We call $V$ the \emph{vertex set} of $\complex$ and write $\vertexset(\complex):=V$.
The \emph{dimension} of $\complex$
is the maximal dimension of its simplices. 
For $\sigma,\tau\in \complex$,
we call $\sigma$ a \emph{coface} of $\tau$ in $\complex$ if $\tau$ is a face of $\sigma$.
In this case, $\sigma$ is a \emph{cofacet} of $\tau$ if their dimensions
differ by exactly one.
A simplicial complex $\subcomplex$ is a \emph{subcomplex} of $\complex$ if $\subcomplex\subseteq\complex$.
Given $\othervertexset\subseteq\vertexset$, the \emph{induced subcomplex} by $\othervertexset$ is the set
of all simplices $\sigma$ in $\complex$ with $\sigma\subseteq\othervertexset$.
For a subcomplex $\subcomplex\subseteq\complex$ and a vertex $v\in\vertexset(\complex)\setminus\vertexset(\subcomplex)$,
we define the join $v\join\subcomplex:=\{v\join\sigma\mid\sigma\in\subcomplex\}$.
For a vertex $v\in \complex$, the \emph{star} of $v$ in $\complex$, denoted by $\ostar(v,\complex)$,
is the set of all cofaces of $v$ in $\complex$. In general, the star is not
a subcomplex, but we can make it a subcomplex by adding all faces of star
simplices, which is denoted by the \emph{closed star} $\cstar(v,\complex)$.
Equivalently, the closed star is the smallest subcomplex of $\complex$ containing
the star of $v$. The \emph{link} of $v$, $\link(v,\complex)$, is defined
as $\cstar(v,\complex)\setminus\ostar(v,\complex)$. It can be checked that the link
is a subcomplex of $\complex$. When the complex is clear from context,
we will sometimes omit the $\complex$ in the notation of stars and links.

\paragraph{Simplicial maps.}
A map $\complex\stackrel{\smap}{\rightarrow}\othercomplex$ between simplicial complexes is called \emph{simplicial}
if with $\sigma=\{v_0,\ldots,v_k\}\in\complex$, $\smap(\sigma)$ is equal to $\{\smap(v_0),\ldots,\smap(v_k)\}$
and $\smap(\sigma)$ is a simplex in $\othercomplex$. 
By definition, a simplicial map
maps vertices to vertices and is completely determined by its action on the vertices.
Moreover, the composition of simplicial maps is again simplicial.

A simple example of a simplicial map is the inclusion map $\subcomplex\stackrel{\smap}{\hookrightarrow}\complex$
where $\subcomplex$ is a subcomplex of $\complex$. If $\complex=\subcomplex\cup\{\sigma\}$
with $\sigma\notin\subcomplex$, we call $\smap$ an \emph{elementary inclusion}.
The simplest example of a non-inclusion simplicial map is $\complex\stackrel{\smap}{\rightarrow}\othercomplex$
such that there exist two vertices $u,v\in\complex$ with $\vertexset(\othercomplex)=\vertexset(\complex)\setminus\{v\}$, 
$\smap(u)=\smap(v)=u$, and $\smap$ is the identity on all remaining vertices of $\complex$.
We call $\smap$ an \emph{elementary contraction} and write $\contr{u}{v}{i}$ as a shortcut.
These notions were introduced by Dey, Fan and Wang in \cite{dfw-computing} and they also showed that
any simplicial map $\complex\stackrel{\smap}{\rightarrow}\othercomplex$
can be written as the composition of elementary 
contractions\footnote{
	They talk about "collapses" instead of "contractions", but this notion clashes with the standard notion of simplicial collapses 
	of free faces that we use later. Therefore, we decided to use "contraction", even though the edge between the contracted vertices 
	might not be present in the complex.
}
and inclusions.

A \emph{tower} of length $m$ is a collection of simplicial complexes $\complex_0,\ldots,\complex_m$ 
and simplicial maps $\smap_i:\complex_i\rightarrow\complex_{i+1}$ for $i=0,\ldots,m-1$.
From this initial data, we obtain simplicial maps $\smap_{i,j}:\complex_i\to\complex_{j}$
for $i\leq j$ by composition, where $\smap_{i,i}$ is simply the identity map on $\complex_i$.
The term ``tower'' is taken from category theory, where it denotes a (directed) path in a category with morphisms
from objects with smaller indices to objects with larger indices.
Indeed, since simplicial complexes form a category with simplicial maps as their morphisms, 
the specified data defines a tower in this category.
A tower is called a \emph{filtration} if all $\smap_i$ are inclusion maps.
The \emph{dimension} of a tower is the maximal dimension among the $\complex_i$,
and the \emph{width} of a tower is the maximal size among the $\complex_i$.
For filtrations, dimension and width are determined by the dimension and size of $\complex_m$,
but this is not necessarily true for general towers.

\paragraph{Homology and Collapses.}
For a fixed base field $\field$, let $H_p(\complex):=H_p(\complex,\field)$ 
the \emph{$p$-dimensional homology group} of $\complex$.
It is well-known that $H_p(\complex)$ is a $\field$-vector space.
Moreover, a simplicial map $\complex\stackrel{\smap}{\rightarrow}\othercomplex$
induces a linear map $H_p(\complex)\stackrel{\smap^\ast}{\rightarrow}H_p(\othercomplex)$.
In categorical terms, the equivalent statement is
that homology is a functor from the category of simplicial complexes
and simplicial maps to the category of vector spaces and linear maps.

We will make use of the following homology-preserving operation:
a \emph{free face} in $\complex$, is a simplex with exactly one proper coface in $\complex$. 
An \emph{elementary collapse} in $\complex$ is the operation of 
removing a free face and its unique coface from $\complex$,
yielding a subcomplex of $\complex$. We say that $\complex$ \emph{collapses to} $\subcomplex$, 
if there is a sequence of elementary collapses 
transforming $\complex$ into $\subcomplex$.
The following result is then well-known:

\begin{lemma} \label{lem:collapse}
	Let $\complex$ be a complex that collapses into the complex $\subcomplex$. 
	Then, the inclusion map $\subcomplex\stackrel{\smap}{\hookrightarrow}\complex$
	induces an isomorphism $\smap_\ast$ between $H_p(\subcomplex)$ and $H_p(\complex)$.
\end{lemma}

\paragraph{Barcodes.}
A \emph{persistence module} is a sequence vector spaces $\vectorspace_0,\ldots,\vectorspace_m$ and linear maps 
$\linearmap_{i,j}:\vectorspace_i\to\vectorspace_j$ for $i<j$ such that $\linearmap_{i,i}=\id_{\vectorspace_i}$
and $\linearmap_{i,k}=\linearmap_{j,k}\circ\linearmap_{i,j}$ for $i \leq k \leq j$.
As primary example, we obtain a persistence module by applying the homology functor on any simplicial tower.
Persistence modules admit a decomposition into indecomposable summands in the following sense.
Writing $I_{b,d}$ with $b\leq d$ for the persistence module
\[
	\underbrace{
		0 \xrightarrow{0} \ldots \xrightarrow{0} 0
	}_{\text{$b - 1$ times}}
	\xrightarrow{0}
	\underbrace{
		\field \xrightarrow{\id} \ldots \xrightarrow{\id} \field
	}_{\text{$d - b + 1$ times}}
	\xrightarrow{0}
	\underbrace{
		0 \xrightarrow{0} \ldots \xrightarrow{0} 0
	}_{\text{$m-d$ times}},
\]
we can write every persistence module as the direct sum $I_{b_1,d_1}\oplus\ldots\oplus I_{b_s,d_s}$, 
where the direct sum of persistence modules is 
defined component-wise for vector spaces and linear maps in the obvious way. Moreover, 
this decomposition is uniquely defined up to isomorphisms and 
re-ordering, thus the pairs $(b_1,d_1),\ldots,(b_s,d_s)$ are an invariant of the persistence module, 
called its \emph{barcode}. When the persistence 
module was generated by a tower, we also talk about the barcode of the tower. 

\paragraph{Matrix reduction.} 
In this paragraph, we assume that $(\complex_i)_{i=0,\ldots,m}$
is a filtration such that $\complex_0=\emptyset$
and $\complex_{i+1}$ has exactly one more simplex than $\complex_i$.
We label the simplices of $\complex_m$ accordingly as $\sigma_1,\ldots,\sigma_m$,
with $\complex_{i}\setminus\complex_{i-1}=\{\sigma_i\}$.
The filtration can be encoded as a \emph{boundary matrix} $\boundarymatrix$ of dimension $m\times m$,
where the $(ij)$-entry is $1$ if $\sigma_i$ is a facet of $\sigma_j$, and 0 otherwise.
In other words, the $j$-th column of $\boundarymatrix$ encodes the facets of $\sigma_j$, 
and the $i$-th row of $\boundarymatrix$ encodes the cofacets of $\sigma_i$.
Moreover, $\boundarymatrix$ is upper-triangular because every $\complex_i$ is a simplicial complex.
We will sometimes identify rows and columns in $\boundarymatrix$ with the corresponding simplex in $\complex_m$.
Adding the $k$-simplex $\sigma_i$ to $\complex_{i-1}$ either introduces one new homology class (of dimension $k$),
or turns a non-trivial homology class (of dimension $k-1$) trivial.
We call $\sigma_i$ and the $i$-th column of $\boundarymatrix$ 
\emph{positive} or \emph{negative}, respectively (with respect to the given filtration).

For the computation of the barcode, we assume for simplicity homology over the base field $\mathbb{Z}_2$, 
and interpret the coefficients of $\boundarymatrix$ accordingly.
In an arbitrary matrix $\arbmatrix$, a \emph{left-to-right column addition} is an operation of the form
$\arbmatrix_k\gets \arbmatrix_k+\arbmatrix_\ell$ with $\ell<k$,
where $\arbmatrix_k$ and $\arbmatrix_\ell$ are columns of the matrix.
The \emph{pivot} of a non-zero column is the largest non-zero index of the corresponding column.
A non-zero entry is called a pivot if its row is the pivot of the column.
A matrix $\rmatrix$ is called a \emph{reduction} of $\arbmatrix$ if $\rmatrix$ is obtained
by a sequence of left-to-right column additions from $\arbmatrix$
and no two columns in $\rmatrix$ have the same pivot. 
It is well-known that, although $\boundarymatrix$ does not have a unique reduction,
the pivots of all its reductions are the same. Moreover, the pivots 
$(b_1,d_1),\ldots,(b_s,d_s)$ of $\rmatrix$ are precisely the barcode of the filtration.
A direct consequence is that a simplex $\sigma_i$ is positive if and only if the $i$-th
column in $\rmatrix$ is zero.

The standard persistence algorithm processes the columns from left to right.
In the $j$-th iteration, as long as the $j$-th column is not empty and
has a pivot that appears in a previous column, it performs a left-to-right column addition.
In this work, we use a simple improvement of this algorithm that is called \emph{compression}:
before reducing the $j$-th column, it first scans through the non-zero entries of the column. If a row index $i$
corresponds to a negative simplex (i.e., if the $i$-th column is not zero at this point in the algorithm),
the row index can be deleted without changing the pivots of the matrix. 
After this initial scan, the column is reduced in the same way as in the standard
algorithm. See~\cite[\S. 3]{chunk} for a discussion
(we remark that this optimization was also used in~\cite{zc-computing}).

\section{From towers to filtrations} \label{sec:from_towers_to_filtrations}

We phrase now our first result which says that any tower
can be converted into a filtration of only marginally larger size
with a space-efficient streaming algorithm:

\begin{theorem}[Conversion Theorem] \label{thm:main_theorem_1}
	Let
	$
		\tower:\,
		\complex_0	\xrightarrow{\smap_0}	
		\complex_1	\xrightarrow{\smap_1}	
		\ldots		\xrightarrow{\smap_{m-1}}	
		\complex_m
	$
	be a tower where, w.l.o.g., $\complex_0 = \emptyset$ and each $\smap_i$ is either an elementary inclusion
	or an elementary contraction. Let $\towerdim$ denote the dimension
	and $\towerwidth$ the width of the tower, and let $n\leq m$ denote the total number
	of elementary inclusions, and $n_0$ the number of vertex inclusions. Then, there exists a filtration
	$
		\filtration:\,
		\hat{\complex}_0	\xhookrightarrow{}	
		\hat{\complex}_1	\xhookrightarrow{}	
		\ldots			\xhookrightarrow{}	
		\hat{\complex}_m
	$,
	where the inclusions are not necessarily elementary, such that $\tower$ and $\filtration$ have the same barcode
	and the width of the filtration $|\hat{\complex}_m|$
	is at most $O(\towerdim \cdot n\log n_0)$. Moreover, $\filtration$ can be computed from $\tower$
	with a streaming algorithm in $O(\towerdim \cdot |\hat{\complex}_m| \cdot C_\towerwidth)$ time
	and space complexity $O(\towerdim \cdot \towerwidth)$,
	where $C_\towerwidth$ is the cost of an operation in a dictionary with $\towerwidth$ elements.
\end{theorem}

The remainder of the section is organized as follows. We define $\filtration$ in Section~\ref{ssec:active_and_small_coning}
and prove that it yields the same barcode as $\tower$ in Section~\ref{ssec:topological_equivalence}.
In Section~\ref{ssec:size_analysis}, we prove the upper bound on the width of the filtration.
In Section~\ref{ssec:algorithm_for_conversion}, we explain the algorithm to compute $\filtration$
and analyze its time and space complexity.

\subsection{Active and small coning} \label{ssec:active_and_small_coning}

\paragraph{Coning.}
We briefly revisit the \emph{coning strategy} introduced by Dey, Fan and Wang~\cite{dfw-computing}. 
Let $\smap:\complex \to \othercomplex$ be an elementary contraction $\contr{u}{v}{i}$ and define
\[
	\othercomplex^\ast = \complex \cup \left( u \join \cstar(v,\complex) \right).
\]
An example is shown in Figure~\ref{fig:coning_strategy}. 

\begin{figure}[h]
	\centering

	\includegraphics[width=129mm]{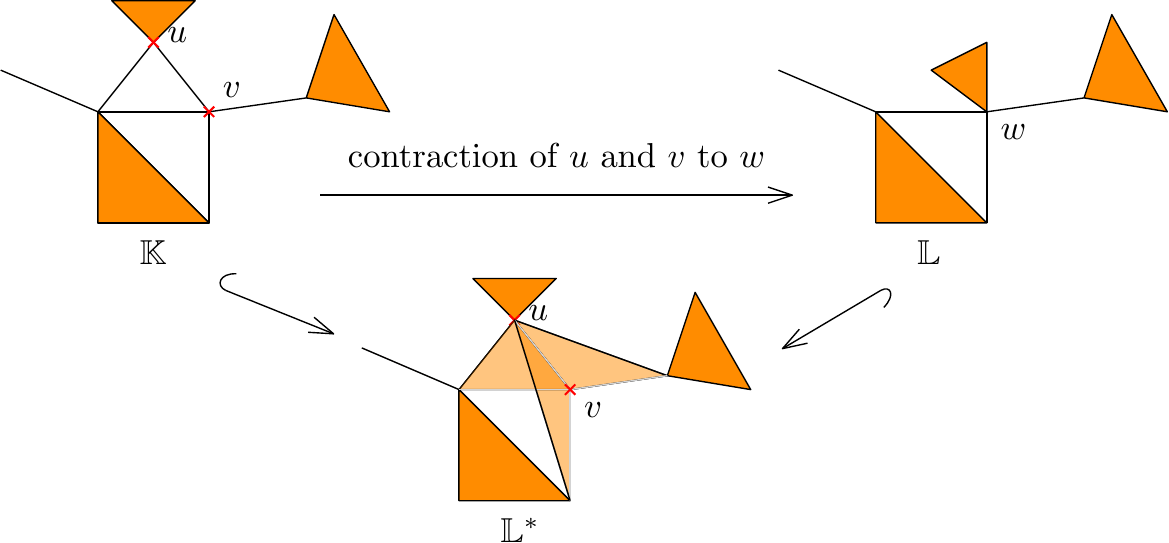}

	\caption{Construction example of $\othercomplex^\ast$, were $u$ and $v$ in $\complex$ are contracted to $w$ in $\othercomplex$} 
	\label{fig:coning_strategy}
\end{figure}

Dey et al.\ show that $\othercomplex \subseteq \othercomplex^\ast$ and 
that the map induced by inclusion is an isomorphism 
between $H(\othercomplex)$ and $H(\othercomplex^\ast)$.
By applying this result at any elementary contraction,
this implies that every zigzag tower can be transformed
into a zigzag filtration with identical barcode.

Given a tower $\tower$, we can also obtain an non-zigzag filtration using coning, if we continue the operation on $\othercomplex^\ast$ instead 
of going back to $\othercomplex$. More precisely,
we set $\tilde{\complex}_0 := \complex_0$ and if $\smap_i$ is an inclusion of simplex $\sigma$,
we set $\tilde{\complex}_{i+1}:=\tilde{\complex}_i\cup\{\sigma\}$. If $\smap_i$ is a contraction
$\contr{u}{v}{i}$, we set $\tilde{\complex}_{i+1} = \tilde{\complex}_i \cup \left( u \join \cstar(v,\tilde{\complex}_i) \right)$.
Indeed, it can be proved that $(\tilde{\complex}_i)_{i=0,\ldots,m}$ has the same barcode as $\tower$. 
However, the filtration will not be small,
and we will define a smaller variant now.

\smallskip

Our new construction yields a sequence of complexes $\hat{\complex}_0,\ldots,\hat{\complex}_{m}$ with 
$\hat{\complex}_i\subseteq\hat{\complex}_{i+1}$.
During the construction, we maintain a flag for each vertex in $\hat{\complex}_i$, which marks the vertex as \emph{active} or \emph{inactive}.
A simplex is called \emph{active} if all its vertices are active, and \emph{inactive} otherwise.
For a vertex $u$ and a complex $\hat{\complex}_i$, let $\activestar(u,\hat{\complex}_i)$ denote its \emph{active closed star}, 
which is the set of active simplices in $\hat{\complex}_i$ in the closed star of $u$.

The construction is inductive, starting with $\hat{\complex}_0 := \emptyset$.
If $\complex_i \stackrel{\smap_i}{\to} \complex_{i+1}$ is an elementary inclusion with $\complex_{i+1} = \complex_i \cup \{\sigma\}$, 
set $\hat{\complex}_{i+1} := \hat{\complex}_i \cup \{\sigma\}$.
If $\sigma$ is a vertex, we mark it as active. 
It remains the case that $\complex_i \stackrel{\smap_i}{\to} \complex_{i+1}$ is an elementary contraction of the vertices $u$ and 
$v$.
If $|\activestar(u,\hat{\complex}_i)| \leq |\activestar(v,\hat{\complex}_i)|$, we set 
\[
	\hat{\complex}_{i+1} = \hat{\complex}_i \cup \left(v\join\activestar(u,\hat{\complex}_i)\right)
\] 
and mark $u$ as inactive. Otherwise, we set
\[
	\hat{\complex}_{i+1} = \hat{\complex}_i \cup \left( u\join \activestar(v,\hat{\complex}_i) \right)
\] 
and mark $v$ as inactive. This ends the description of the construction. We write $\filtration$ for the filtration 
$(\hat{\complex}_i)_{i=0,\ldots,m}$.

There are two major changes compared to the construction of $(\tilde{\complex}_i)_{i=0,\ldots,m}$. 
First, to counteract the potentially large growth of the involved cones, we restrict 
coning to active simplices. We will show below that the subcomplex of $\hat{\complex}_i$ induced by the active vertices is isomorphic to 
$\complex_i$. As a consequence, we add the same number of simplices by passing from $\hat{\complex}_i$ to $\hat{\complex}_{i+1}$ 
as in the approach by Dey et al.\ does when passing from $\complex$ to $\othercomplex^\ast$.

A second difference is that our strategy exploits that an elementary contraction of two vertices $u$ and $v$ leaves us with a choice: 
we can either take $u$ or $v$ as the representative of the contracted vertex. In terms of simplicial maps, 
these two choices correspond to setting $\smap_i(u) = \smap_i(v) = u$ or $\smap_i(u) = \smap_i(v) = v$, 
if $\smap_i$ is the elementary contraction of $u$ and $v$. It is obvious that both choices yield identical complexes $\complex_{i+1}$ 
up to renaming of vertices. However, the choices make a difference in terms of the size of $\hat{\complex}_{i+1}$, 
because the  active closed star of $u$ to $v$ in $\hat{\complex}_i$ might differ in size. Our construction simply choose the 
representative which causes the smaller $\hat{\complex}_{i+1}$.

\subsection{Topological equivalence} \label{ssec:topological_equivalence}

We make the following simplifying assumption for $\tower$. Let $\smap_i$ be an elementary contraction of $u$ and $v$. If our construction 
of $\hat{\complex}_{i+1}$ turns $v$ inactive, we assume that $\smap_i(u) = \smap_i(v) = u$. 
Otherwise, we assume $\smap_i(u) = \smap_i(v) = v$. 
Indeed, this is without loss of generality because it corresponds to a renaming of the simplices in each $\complex_i$ and yields equivalent 
persistence modules. The advantage of this convention is the following property, which follows from a straight-forward inductive argument.

\begin{lemma}\label{lem:active_vertices}
	For every $i$ in $\{0,...,m\}$, the set of vertices of $\complex_i$ is equal to the set of active vertices in $\hat{\complex}_i$.
\end{lemma}

This allows us to interpret $\complex_i$ and $\hat{\complex}_i$ as simplicial complexes defined over a common vertex set. 
In fact, $\complex_i$ is the subcomplex of $\hat{\complex}_i$ induced by the active vertices:

\begin{lemma} \label{lem1}
	A simplex $\sigma$ is in $\complex_i$ if and only if $\sigma$ is an active simplex in $\hat{\complex}_i$.
\end{lemma}

\begin{proof}
	We use induction on $i$. The statement is true for $i=0$, because $\complex_0=\emptyset=\hat{\complex}_0$.
	So assume first $\smap_i: \complex_i \to \complex_{i+1}$ is an elementary inclusion that adds a $d$-simplex $\sigma=(v_0,\ldots,v_d)$ to 
	$\complex_{i+1}$.
	If $\sigma$ is a vertex, it is set active in $\hat{\complex}_{i+1}$ by construction. Otherwise, $v_0,\ldots,v_d$ are active
	in $\hat{\complex}_i$ by induction and stay active in $\hat{\complex}_{i+1}$. In any case, $\sigma$ is active in $\hat{\complex}_{i+1}$.
	The equivalence for the remaining simplices is straight-forward.
	
	If $\smap_i$ is an elementary contraction $\contr{u}{v}{i}$, we prove both directions of the equivalence separately. For ``$\Rightarrow$'', 
	fix a $d$-simplex $\sigma \in \complex_{i+1}$. It suffices to show that $\sigma \in \hat{\complex}_{i+1}$, as in this case, it is also 
	active by Lemma~\ref{lem:active_vertices}. If $\sigma \in \complex_{i}$, this follows at once by induction because 
	$\complex_i \subseteq \hat{\complex}_i \subseteq \hat{\complex}_{i+1}$. If $\sigma \notin \complex_{i}$, $u$ must be a vertex of $\sigma$. 
	Moreover, writing $\sigma = \{u, v_1, ..., v_d\}$ and $\sigma' = \{v,v_1,\ldots,v_d\}$ we have that $\sigma' \in \complex_i$ and 
	$\smap_i(\sigma') = \sigma$. In particular, the vertices $v_1,\ldots,v_d$ are active in $\hat{\complex}_i$ by induction, hence 
	$\{v_1,\ldots,v_d\}$ is in the active closed star of $v$ in $\hat{\complex}_i$. By construction, $\{u, v_1, ..., v_d\}=\sigma$ is in 
	$\hat{\complex}_{i+1}$.
	
	For ``$\Leftarrow$'', let $\sigma \in \hat{\complex}_{i+1}\setminus \complex_{i+1}$. We show that $\sigma$ is an inactive simplex in 
	$\hat{\complex}_{i+1}$. By Lemma~\ref{lem:active_vertices}, this is equivalent to show that $\sigma$ contains a vertex not in 
	$\complex_{i+1}$.
	
	Case 1: $\sigma \in \hat{\complex}_{i}$. 
	If $\sigma$ is inactive in $\hat{\complex}_{i}$, it stays inactive in $\hat{\complex}_{i+1}$. So, assume that $\sigma$ is active in 
	$\hat{\complex}_i$ and thus $\sigma \in \complex_i$ by induction. But $\sigma \notin \complex_{i+1}$, so $\sigma$ must have $v$ as a vertex 
	and is therefore inactive in $\hat{\complex}_{i+1}$.
	
	Case 2: $\sigma\in\hat{\complex}_{i+1}\setminus\hat{\complex}_i$. 
	By construction, $\sigma$ is of the form $\{u,v_1,\ldots,v_d\}$ such that $\{v_1,\ldots,v_d\}$ is in the active closed star of $v$ in 
	$\hat{\complex}_i$. Assume for a contradiction that $v\neq v_j$ for all $j=1,\ldots,d$. Then, $\sigma'=\{v,v_1,\ldots,v_d\}$ is active in 
	$\hat{\complex}_i$ and thus, by induction, a simplex in $\complex_i$. But then, $\smap_i(\sigma')=\sigma\in \complex_{i+1}$ which is a 
	contradiction to our choice of $\sigma$. It follows that $v$ is a vertex of $\sigma$ which proves our claim.
\end{proof}

\begin{lemma} \label{lem3}
	For every $0 \leq i \leq m$, the complex $\hat{\complex}_i$ collapses to $\complex_i$.
\end{lemma}

\begin{proof}
	We use induction on $i$. For $i = 0$, $\complex_0 = \hat{\complex}_0$, and the statement is trivial. Suppose that the statement holds for 
	$\hat{\complex}_i$ and $\complex_i$ using the sequence $s_i$ of elementary collapses. Note that these collapses only concern inactive 
	simplices in $\hat{\complex}_i$. For an inactive vertex $v \in \hat{\complex}_i$, the construction of $\hat{\complex}_{i+1}$ ensures that 
	$v$ does not gain any additional coface. This implies that $s_i$ is still a sequence of elementary collapses for $\hat{\complex}_{i+1}$, 
	yielding a complex $\hat{\complex}^*_{i+1}$ with $\complex_{i+1} \subseteq \hat{\complex}^*_{i+1} \subseteq \hat{\complex}_{i+1}$. In 
	particular, $\hat{\complex}^*_{i+1}$ only contains vertices that are still active in $\hat{\complex}_i$. If $\smap_i$ is an elementary 
	inclusion, $\hat{\complex}^*_{i+1} = \complex_{i+1}$, because any all vertices in $\hat{\complex}_i$ remain active in 
	$\hat{\complex}_{i+1}$. For $\smap_i$ being an elementary contraction $\contr{u}{v}{i}$, set 
	$S := \hat{\complex}^*_{i+1} \setminus \complex_{i+1}$ as the remaining set of simplices that still need to be collapsed to obtain 
	$\complex_{i+1}$. All simplices of $S$ have $v$ as vertex. More precisely, $S$ is the set of all simplices of the form 
	$\{v,v_1,\ldots,v_d\}$ with $v_1,\ldots,v_d$ active in $\hat{\complex}_{i+1}$. We split $S = S_u \cup S_{\neg u}$ where $S_u \subset S$ are 
	the simplices in $S$ that contain $u$ as vertex, and $S_{\neg u} = S \setminus S_u$.
	
	We claim that the mapping that sends $\{u,v,v_1,\ldots,v_d\} \in S_u$ to $\{v,v_1,\ldots,v_d\} \in S_{\neg u}$ is a bijection. This map is 
	clearly injective. If $\sigma = \{v,v_1,\ldots,v_d\} \in S_{\neg u}$, then $\sigma \in \hat{\complex}_i$ (because every newly added simplex 
	in $\hat{\complex}_{i+1}$ contains $u$). Also, $\sigma \in \hat{\complex}^*_{i+1}$, and is therefore active in $\hat{\complex}_i$. By 
	construction, $\{u,v,v_1,\ldots,v_d\} \in \hat{\complex}_{i+1}$, proving surjectivity.
	
	We now define a sequence of elementary collapses from $\hat{\complex}^*_{i+1}$ to $\complex_{i+1}$. Choose a simplex 
	$\sigma = \{v,v_1,\ldots,v_d\} \in S_{\neg u}$ of maximal dimension, and let $\tau = \{u, v, v_1, \ldots, v_d\}$ denote the corresponding 
	simplex in $S_u$. Then $\sigma$ is indeed a free face in $\hat{\complex}^*_{i+1}$, because if there was another coface $\tau' \neq \tau$, it 
	takes the form $\{w,v,v_1,\ldots,v_d\}$ with $w \neq u$ active. So, $\tau' \in S_{\neg u}$, and $\tau'$ has larger dimension than $\sigma$, 
	a contradiction. Therefore, the pair $(\sigma,\tau)$ defines an elementary collapse in $\hat{\complex}^*_{i+1}$. We proceed with this 
	construction, always collapsing a remaining pair in $S_{\neg u} \times S_u$ of maximal dimension, until all elements of $S$ have been 
	collapsed.
\end{proof}

\begin{proposition} \label{prop:barcode_equiv}
	$\tower$ and $\filtration$ have the same barcode.
\end{proposition}

\begin{proof}
	Let $\hat{\smap}_i: \hat{\complex}_i \to \hat{\complex}_{i+1}$ denote the inclusion map from $\hat{\complex}_i$ to $\hat{\complex}_{i+1}$. 
	By combining Lemma~\ref{lem3} with Lemma~\ref{lem:collapse}, we have an isomorphism $\inc_i^*: H(\complex_i) \to H(\hat{\complex}_i)$, for 
	all $0 \leq i \leq m$, induced by the inclusion maps $\inc_i: \complex_i \to \hat{\complex}_i$, and therefore the following diagram 
	connecting the persistence modules induced by $\tower$ and $\filtration$:
	\begin{equation}\label{dgmH}
		\begin{CD}
			H(\complex_0)		@>\smap^\ast_0>>	H(\complex_1)		@>\smap^\ast_1>>
				...	@>\smap^\ast_{m-1}>>	H(\complex_m)\\
			@VV\inc_0^*V					@VV\inc_1^*V
				@.				@VV\inc_m^*V\\
			H(\hat{\complex}_0)	@>\hat{\smap}^*_0>>	H(\hat{\complex}_1)	@>\hat{\smap}^*_1>>
				...	@>\hat{\smap}^*_{m-1}>>	H(\hat{\complex}_m)
		\end{CD}
	\end{equation}
	The \emph{Persistence Equivalence Theorem}~\cite[p.159]{edelsbrunnerharer} 
        asserts that $(\complex_j)_{j}$ and $(\hat{\complex}_j)_{j}$, with $j = 0, ..., m$, have the same barcode 
        if \eqref{dgmH} commutes, that is, if $\inc_{i+1}^* \circ \smap_i^* = \hat{\smap}_i^* \circ \inc_i^*$, for all $0 \leq i < m$.
	
	Two simplicial maps $\smap: \complex \to \othercomplex$ and $\othersmap: \complex\to\othercomplex$ are \emph{contiguous} if, for all 
	$\sigma \in \complex$, $\smap(\sigma) \cup \othersmap(\sigma)\in\othercomplex$. Two contiguous maps are known to be homotopic~\cite[Theorem 
	12.5.]{munkres} and thus equal at homology level, that is, $\smap^\ast = \othersmap^\ast$. We show that $\inc_{i+1} \circ \smap_i$  and 
	$\hat{\smap}_i \circ \inc_i$ are contiguous. This implies that \eqref{dgmH} commutes, because, by functoriality, 
	$\inc_{i+1}^* \circ \smap_i^* = (\inc_{i+1} \circ \smap_i)^* = (\hat{\smap}_i \circ \inc_i)^* = \hat{\smap}_i^* \circ \inc_i^*$.
	
	To show contiguity, fix $\sigma \in \complex_i$ and observe that $(\hat{\smap}_i \circ \inc_i)(\sigma) = \sigma$ because $\hat{\smap}_i$ 
	and $\inc_i$ are inclusions. If $\smap_i(\sigma) = \sigma$, $(\inc_{i+1} \circ \smap_i)(\sigma) = \sigma$ as well, and the contiguity 
	condition is clearly satisfied. So, let $\smap_i(\sigma) \neq \sigma$. Then $\smap_i$ is an elementary contraction $\contr{u}{v}{i}$, and 
	$\sigma$ is of the form $\{v,v_1,\ldots,v_d\}$, where one of the $v_j$ might be equal to $u$. Then, 
	$(\inc_{i+1} \circ \smap_i)(\sigma) = \{u, v_1, ..., v_d\}$. Consequently, 
	$(\inc_{i+1} \circ \smap_i)(\sigma) \cup (\hat{\smap}_i \circ \inc_i)(\sigma) = \{u, v, v_1, ..., v_d\}$. By Lemma~\ref{lem1}, 
	$\sigma = \{v,v_1,\ldots,v_d\}$ is in the active closed star of $v$ in $\hat{\complex}_i$, and by construction 
	$\{u, v, v_1, ..., v_d\} \in \hat{\complex}_{i+1}$, which proves contiguity of the maps.
\end{proof}

\subsection{Size analysis}\label{ssec:size_analysis}

\paragraph{The contracting forest.}
We associate a rooted labeled forest $\forest_j$ to a prefix 
$\emptyset = \complex_0 \xrightarrow{\smap_0} \ldots \xrightarrow{\smap_{j-1}} \complex_j$ of $\tower$ inductively as follows: 
For $j = 0$, $\forest_0$ is the empty forest. 
Let $\forest_{j-1}$ be the forest of $\complex_0 \rightarrow \ldots \rightarrow \complex_{j-1}$.
If $\smap_{j-1}$ is an elementary inclusion of a $d$-simplex, we have two cases: if $d > 0$, set $\forest_j:=\forest_{j-1}$. 
If a vertex $v$ is included, $\forest_j:=\forest_{j-1} \cup \{\node\}$, with $\node$ a single node tree labeled with $v$.
If $\smap_{j-1}$ is an elementary contraction contracting two vertices $u$ and $v$ in $\complex_{j-1}$, 
there are two trees in $\forest_{j-1}$, whose roots are labeled $u$ and $v$. In $\forest_j$, these two trees are merged 
by making their roots children of a new root, which is labeled with the vertex that $u$ and $v$ are mapped to. 

We can read off from the construction immediately that the roots of $\forest_i$ are labeled with the vertices of the complex $\complex_i$, 
for every $i=0,\ldots,m$. 
Moreover, each leaf corresponds to the inclusion of a vertex in $\complex_0 \rightarrow \ldots \rightarrow \complex_i$, and each internal node 
corresponds to a contraction of two vertices. In particular, $\forest_i$ is a \emph{full} forest, that is, every node has $0$ or $2$ children.

Let $\forest:=\forest_m$ denote the forest of the tower $\tower$. 
Let $\Sigma$ denote the set of all simplices that are added at elementary inclusions in $\tower$, and recall that $n=|\Sigma|$.
A $d$-simplex $\sigma \in \Sigma$ added by $\smap_i$ is formed by $d+1$ vertices, which correspond to $d+1$ roots of $\forest_{i+1}$, and 
equivalently, to $d+1$ nodes in $\forest$.
For a node $\node$ in $\forest$, we denote by $E(\node) \subseteq \Sigma$ the subset of simplices with at least one vertex that appears
as label in the subtree of $\forest$ rooted at $\node$.
If $\nodechildone$ and $\nodechildtwo$ are the children of $\node$, $E(\nodechildone)$ and $E(\nodechildtwo)$ are both subsets of $E(\node)$, but 
not disjoint in general. However, the following relation follows at once:
\begin{eqnarray}\label{eqn:E-equation}
	|E(\node)| & \geq & |E(\nodechildone)| + |E(\nodechildtwo)\setminus E(\nodechildone)|
\end{eqnarray}
We say that the set $\nodeset$ of nodes in $\forest$ is \emph{independent}, if there are no two nodes $\node_1\neq \node_2$ in $\nodeset$, such that 
$\node_1$ is an ancestor of $\node_2$ in $\forest$.
A vertex in $\complex_i$ appears as label in at most one $\forest$-subtree rooted at a vertex in the independent set $\nodeset$. 
Thus, a $d$-simplex $\sigma$ can only appear in up to $d+1$ $E$-sets of vertices in $\nodeset$. That implies:

\begin{lemma}\label{lem:indep_E_size}
	Let $\nodeset$ be an independent set of vertices in $\forest$. Then,
	\[
		\sum_{\node\in N} |E(\node)| \leq (\towerdim + 1) \cdot n,
	\]
\end{lemma}

\paragraph{The cost of contracting.} 
Recall that a contraction $\complex_i \stackrel{\smap_i}{\to} \complex_{i+1}$ yields an inclusion 
$\hat{\complex}_i \hookrightarrow \hat{\complex}_{i+1}$ that potentially adds more than one simplex. 
Therefore, in order to bound the total size of $\hat{\complex}_m$, we need to bound the number of simplices added in all these contractions.

We define the \emph{cost} of a contraction $\smap_i$ as $|\hat{\complex}_{i+1} \setminus \hat{\complex}_i|$, that is, the number of simplices added 
in this step. Since each contraction corresponds to a node $\node$ in $\forest$, we can associate these costs to the internal nodes in the forest, 
denoted by $c(\node)$. The leaves get cost $0$.

\begin{lemma}\label{lem:vertex-cost-lemma}
	Let $\node$ be an internal node of $\forest$ with children $\nodechildone$, $\nodechildtwo$. 
	Then, $c(\node) \leq 2 \cdot |E(\nodechildone) \setminus E(\nodechildtwo)|$.
\end{lemma}

\begin{proof}
	Let $\smap_i: \complex_i \to \complex_{i+1}$ denote the contraction that is represented by the node $\node$, and let $w_1$ and $w_2$ the 
	labels of its children $\nodechildone$ and $\nodechildtwo$, respectively. By construction, $w_1$ and $w_2$ are vertices in $\complex_i$ that 
	are contracted by $\smap_i$. Let $C_1 = \cstar(w_1,\complex_i) \setminus \cstar(w_2,\complex_i)$ and 
	$C_2 = \cstar(w_2,\complex_i) \setminus \cstar(w_1,\complex_i)$. By Lemma~\ref{lem1}, 
	$\cstar(w_1,\complex_i)=\activestar(w_1,\hat{\complex}_i)$, and the same for $w_2$. So, because the simplices the two active closed stars 
	have in common will not influence the cost of the contraction, we have,
	\[
		c(\node) \leq \min\{|C_1|,|C_2|\}.
	\]
	
	In particular, $c(\node) \leq |C_1|$. We will show that $|C_1| \leq 2 \cdot |E(\nodechildone) \setminus E(\nodechildtwo)|$.
	
	For every $d$-simplex $\sigma \in \complex_i$, there is some $d$-simplex $\tau \in \Sigma$ that has been added in an elementary inclusion 
	$\smap_j$ with $j < i$, such that $\smap_{i-1} \circ \smap_{i-2} \circ \ldots \circ \smap_{j+1}(\tau) = \sigma$. We call $\tau$ an 
	\emph{origin} of $\sigma$. There might be more than one origin of a simplex, but two distinct simplices in $\complex_i$ cannot have a common 
	origin. Moreover, for every vertex $v$ of $\sigma$, the tree in $\forest_i$ whose root is labeled with $v$ contains exactly one vertex $v'$ 
	of $\tau$ as label. We omit the proof which works by simple induction.
	
	We prove the inequality by a simple charging argument. For each vertex in $C_1$, we charge a simplex in 
	$|E(\nodechildone) \setminus E(\nodechildtwo)|$ such that no simplex is charged more than twice.
	Note that $C_1=(\ostar(w_1,\complex_i)\cup\link(w_1,\complex_i))\setminus\cstar(w_2,\complex_i)$.
	If $\sigma \in \ostar(w_1,\complex_i)$, then fix an origin $\tau$ of $\sigma$.
	Then $\tau$ has a vertex that is a label in the subtree rooted at $\nodechildone$, so $\tau \in E(\nodechildone)$. At the same time, since 
	$w_2$ is not a vertex of $\sigma$, $\tau$ has no vertex in the subtree rooted at $\nodechildtwo$, so $\tau \notin E(\nodechildtwo)$. We 
	charge $\tau$ for the existence of $\sigma$. Because different simplices have different origins, every element in 
	$E(\nodechildone) \setminus E(\nodechildtwo)$ is charged at most once for $\ostar(w_1,\complex_i)$.
	If $\sigma \in \link(w_1,\complex_i)$, $\sigma':=w_1\join \sigma \in \ostar(w_1,\complex_i)$, and we can choose an origin $\tau'$ of 
	$\sigma'$.
	As before, $\tau'\in E(\nodechildone) \setminus E(\nodechildtwo)$ and we charge $\tau'$ for the existence of $\sigma$.
	Again, each element in $E(\nodechildone) \setminus E(\nodechildtwo)$ is charged at most once among all elements in the link.
	This proves the claim.
\end{proof}

An \emph{ascending path} $(\node_1, ..., \node_L)$, with $L \geq 1$, is a path in a forest such that $\node_{i+1}$ is the parent of $\node_i$, 
for $1 \leq i < L$. We call $L$ the \emph{length} of the path and $\node_L$ its \emph{endpoint}. 
For ascending paths in $\forest$, the \emph{cost} of the path is the sum of the costs of the nodes. 
We say that the set $P$ of ascending paths is \emph{independent}, if the endpoints in $P$ are pairwise different and form an independent set of 
nodes. We define the \emph{cost} of $P$ as the sum of the costs of the paths in $P$.

\begin{lemma}\label{lem:path-cost-lemma}
	An ascending path with endpoint $\node$ has cost at most $2 \cdot |E(\node)|$. An independent set of ascending paths in $\forest$
        has cost at most $2 \cdot (\towerdim + 1) \cdot n$.
\end{lemma}

\begin{proof}
	For the first statement, let $p = (\node_1, ..., \node_L)$ be an ascending path with $v_L=v$. 
	Without loss of generality, we can assume the the path starts with a leaf $\node_1$, because otherwise, 
	we can always extend the path to a longer path with at least the same cost.
	We let $p_i = (\node_1,\ldots,\node_i)$ denote the sub-path 
	ending at $\node_i$, for $i = 1,\ldots,L$, so that $p_L = p$. We let $c(p_i)$ denote the cost of the path $p_i$ and show by induction that 
	$c(p_i) \leq 2 \cdot |E(\node_i)|$. For $i = 1$, this follows because $c(p_1) = 0$. For $i = 2,\ldots,L$, $\node_i$ is an internal node, and 
	its two children are $\node_{i-1}$ and some other node $\node'_{i-1}$. Using induction and Lemma~\ref{lem:vertex-cost-lemma}, we have that
	\[
		c(p_i) = c(p_{i-1}) + c(\node_i) 
			\leq 2 \cdot (|E(\node_{i-1})| + |E(\node'_{i-1}) \setminus E(\node_{i-1})|) 
			\leq 2 \cdot |E(\node_i)|,
	\]
	where the last inequality follows from~\eqref{eqn:E-equation}. The second statement follows from 
	Lemma \ref{lem:indep_E_size} because the endpoints of the paths form an independent set in $\forest$.
\end{proof}

\paragraph{Ascending path decomposition.}
An \emph{only-child-path} in a binary tree is an ascending path starting in a leaf and ending at the first encountered node that has a sibling, or 
at the root of the tree. An only-child-path can have length $1$, if the starting leaf has a sibling already. Examples of only-child-paths are shown  
in Figure~\ref{fig:onlychildpaths}. We observe that no node with two children lies on any only-child-path, which implies that the set of 
only-child-paths forms an independent set of ascending paths.

Consider the following pruning procedure for a full binary forest $\forest$. Set $\forest_{(0)} \gets \forest$. In iteration $i$, we obtain the 
forest $\forest_{(i)}$ from $\forest_{(i-1)}$ by deleting the only-child-paths of $\forest_{(i-1)}$. We stop when $\forest_{(i)}$ is empty; this 
happens eventually because at least the leaves of $\forest_{(i-1)}$ are deleted in the $i$-th iteration. Because we start with a full forest, the 
only-child-paths in the first iteration are all of length $1$, and consequently $\forest_{(1)}$ arises from $\forest_{(0)}$ by deleting the leaves 
of $\forest_{(0)}$. Note that the intermediate forests $\forest_{(1)},\forest_{(2)},\ldots$ are not full forests in general. 
Figure~\ref{fig:onlychildpaths} shows the pruning procedure on an example.

\begin{figure}[h]
	\centering

	\includegraphics[width=84mm]{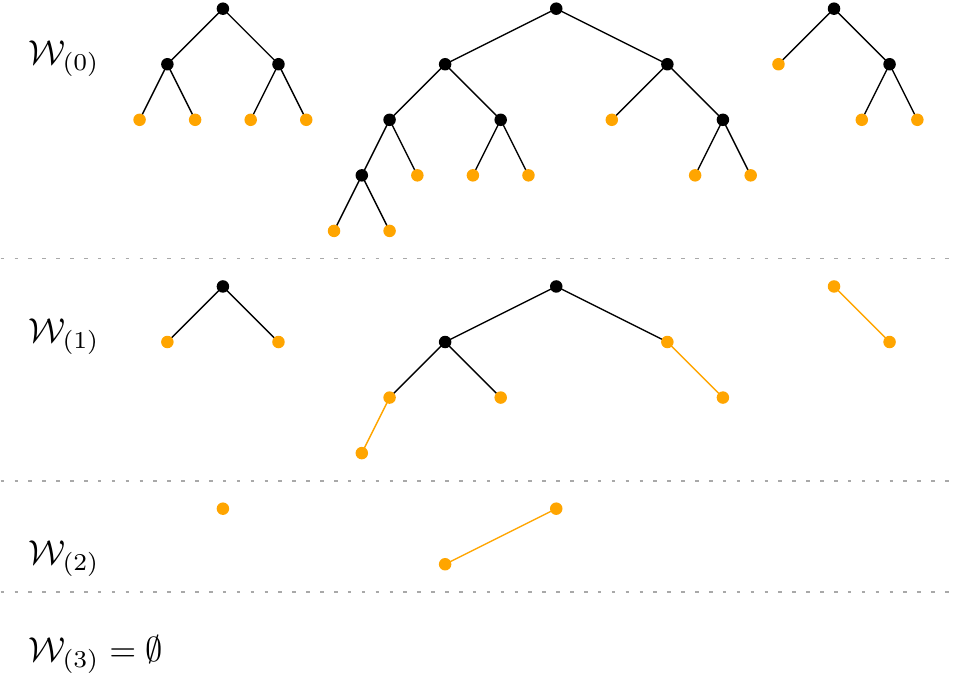}

	\caption{Iterations of the pruning procedure; the only-child-paths are marked in color} \label{fig:onlychildpaths}
\end{figure}

To analyze this pruning procedure in detail, we define the following integer valued function for nodes in $\forest$:
\[
	r(\node)=
	\begin{cases}
		1, 
			& \text{ if $\node$ is a leaf} \\
		r(\nodechildone) + 1, 
			& \text{ if $\node$ has children $\nodechildone,\nodechildtwo$ and $r(\nodechildone) = r(\nodechildtwo)$} \\
		\max \{r(\nodechildone), r(\nodechildtwo)\}, 
			& \text{ if $\node$ has children $\nodechildone,\nodechildtwo$ and $r(\nodechildone) \neq r(\nodechildtwo)$}
	\end{cases}
\]

\begin{lemma} \label{lem12}
	A node $\node$ of a full binary forest $\forest$ is deleted in the pruning procedure during the $r(\node)$-th iteration.
\end{lemma}

\begin{proof}
	We prove the claim by induction on the tree structure. If $\node$ is a leaf, it is removed in the first iteration, and $r(\node)=1$. If 
	$\node$ is an internal node with children $\nodechildone$ and $\nodechildtwo$, let $r_1:=r(\nodechildone)$ and $r_2:=r(\nodechildtwo)$. 
	By induction, $\nodechildone$ is deleted in the $r_1$-th iteration and $\nodechildtwo$ is deleted in the $r_2$-th iteration. There are two 
	cases: if $r_1 = r_2$, $\node$ still has two children after $r_1 - 1$ iterations. This implies that $\node$ does not lie on an only-child-
	path in the forest $\forest_{(r_1 - 1)}$ and is therefore not deleted in the $r_1$-th iteration. But because its children are deleted, 
	$\node$ is a leaf in $\forest_{(r_1)}$ and therefore deleted in the $(r_1 + 1)$-th iteration. 
	It remains the second case that $r_1 \neq r_2$. Assume without loss of generality that $r_1 > r_2$, so that $r(\node) = r_1$. In iteration 
	$r_1$, $\nodechildone$ lies on an only-child-path in $\forest_{(r_1 - 1)}$. Because $\nodechildtwo \notin \forest_{(r_1 - 1)}$, 
	$\nodechildone$ has no sibling in $\forest_{(r_1-1)}$, so the only-child-path extends to $\node$. Consequently, $\node$ is deleted in the 
	$r_1$-th iteration.
\end{proof}

\begin{lemma} \label{lem:r_value_bound}
	For a node $\node$ in a full binary forest, let $s(\node)$ denote the number of nodes in the subtree rooted at $\node$. 
	Then $s(\node) \geq 2^{r(\node)}-1$. In particular, $r(\node) \leq \log_2 (s(\node)+1)$.
\end{lemma}

\begin{proof}
	We prove the claim by induction on $r(\node)$. Note that $r(\node) = 1$ if and only if $v$ is a leaf, which implies the statement for 
	$r(\node) = 1$. For $r(\node) > 1$, it is sufficient to prove the statement assuming that $v$ has a minimal $s$-value among all nodes with 
	the same $r$-value. Since $\node$ is not a leaf, it has two children $\nodechildone$ and $\nodechildtwo$. They satisfy 
	$r(\nodechildone) = r(\nodechildtwo) = r(\node) - 1$, because otherwise $r(\node) = r(\nodechildone)$ or $r(\node) = r(\nodechildtwo)$, 
	contradicting the minimality of $\node$. By induction hypothesis, we obtain that
	\[
		s(\node) = 1 + s(\nodechildone) + s(\nodechildtwo) 
			\geq 1 + \left(2^{r(\nodechildone)} - 1 \right) + \left(2^{r(\nodechildtwo)} - 1 \right) 
			= 2^{r(\node)} - 1.
	\]
\end{proof}

\begin{proposition} \label{prop12}
	$|\hat{\complex}_m| \leq n + 2 \cdot (\towerdim + 1) \cdot n \cdot (1 + \log_2(n_0)) = O(n \cdot \towerdim \cdot \log_2(n_0))$, 
	where $n_0$ is the number of vertices included in $\tower$.
\end{proposition}

\begin{proof}
	Applying the pruning procedure to the contraction forest $\forest$ of $\tower$, we obtain in every iteration a set of independent ascending 
	paths of $\forest$, and the cost of this set is bounded by $2 \cdot (\towerdim + 1) \cdot n$ with Lemma~\ref{lem:path-cost-lemma}.
	Because $\forest$ has at most $2 \cdot n_0 - 1$ nodes, any node $\node$ satisfies $r(\node) \leq \log_2(2 \cdot n_0)$ by 
	Lemma~\ref{lem:r_value_bound}. It follows that the pruning procedure ends after $1 + \log_2(n_0)$ iterations, so the total cost of the 
	contraction forest is at most $2 \cdot (\towerdim + 1) \cdot n \cdot (1 + \log_2(n_0))$. By definition, this cost is equal to the number of 
	simplices added in all contraction steps. Together with the $n$ simplices added in inclusion steps, the bound follows.
\end{proof}

\subsection{Algorithm}\label{ssec:algorithm_for_conversion}

We will make frequent use of the following concept: a \emph{dictionary}
is a data structure that stores a set of \emph{items} of the form \texttt{(k,v)},
where \texttt{k} is called the \emph{key} and \texttt{v} is called
the \emph{value} of the item. 
We assume that all keys stored in the dictionary are pairwise different.
The dictionary support three operations:
\texttt{insert(k,v)} adds a new item to the dictionary, \texttt{delete(k)}
removes the item with key \texttt{k} from the dictionary (if it exists),
and \texttt{search(k)} returns the item with key \texttt{k},
or returns that no such item exists. Common realizations
are balanced binary search trees~\cite[\S12]{cormen} and hash tables~\cite[\S11]{cormen}.

\paragraph{Simplicial complexes by dictionaries.}
The main data structure of our algorithm is a dictionary $\dict$ that represents
a simplicial complex. Every item stored in the dictionary represents a simplex,
whose key is the list of its vertices. Every simplex $\sigma$ itself stores an dictionary
$\innerdict_\sigma$.
Every item in $\innerdict_\sigma$ is a pointer to another item in $\dict$,
representing a cofacet $\tau$ of $\sigma$. The key of the item is
a vertex identifier (e.g., an integer) for $v$ such that $\tau=v\join\sigma$.

How large is $\dict$ for a simplicial complex $\complex$ with $n$ simplices
of dimension $\towerdim$? 
Observe that $\dict$ stores $n$ items, and each key is of size $\leq\towerdim+1$.
That yields a size of $O(n\towerdim)$ plus the size of all $\innerdict_\sigma$.
Since every simplex is the cofacet of at most $\towerdim$ simplices, the size of all these
inner dictionaries is also bounded by $O(n\towerdim)$ 
(assuming that the size of a dictionary is linear in the number of stored elements).

We can insert and delete simplices efficiently in $\dict$ using
dictionary operations. For instance, to insert a simplex $\sigma$
given as a list of vertices, we insert a new item in $\dict$
with the key. Then we search for the $\towerdim$ facets of $\sigma$
(using dictionary search in $\dict$), and notify each facet $\tau$
of $\sigma$ about the insertion by adding a pointer to $\sigma$ to $\innerdict_\tau$, 
using the vertex $\sigma\setminus\tau$ as the key. 
The deletion procedure works similarly. 
Each simplex insertion and deletion requires $O(\towerdim)$ dictionary operations.
In what follows, it is convenient to assume that dictionary operations
have unit costs; we multiply the time complexity
with the cost of a dictionary operation at the end
to compensate for this simplification.

\paragraph{The conversion algorithm.}
We assume that the tower $\tower$ is given to us as a stream where each element
represents a simplicial map $\smap_i$ in the tower.
Specifically, an element starts with a token \{\texttt{INCLUSION}, \texttt{CONTRACTION}\}
that specifies the type of the map. In the first case, the token is followed
by a non-empty list of vertex identifiers specifying the vertices of the simplex to be added.
In the second case, the token is followed by two vertex identifiers $u$ and $v$,
specifying a contraction of type $\contr{u}{v}{i}$.

The algorithm outputs a stream of simplices specifying the filtration
$\filtration$. Specifically, while handling the $i$-th input element,
it outputs the simplices of $\hat{\complex}_{i+1}\setminus\hat{\complex}_i$
in increasing order of dimension (to ensure that every prefix is a simplicial complex).
For simplicity, we assume that output simplices are also specified by a list
of vertices~-- the algorithm can easily be adapted to return the boundary matrix
of the filtration in sparse list representation with the same complexity bounds.

We use an initially empty dictionary $\dict$ as above, 
and maintain the invariant that after the $i$-th iteration,
$\dict$ represents the active subcomplex of $\hat{\complex}_i$,
which is equal to $\complex_i$ by Lemma~\ref{lem1}.

If the algorithm reads an inclusion of a simplex $\sigma$ from the stream, 
it simply adds $\sigma$ to $\dict$ (maintaining the invariant)
and writes $\sigma$ to the output stream.

If the algorithm reads a contraction of two vertices $u$ and $v$,
from $\complex_i$ to $\complex_{i+1}$,
we let $c_i$ denote the cost of the contraction, that is,
$c_i=|\hat{\complex}_{i+1}\setminus\hat{\complex}_i|$.
The first step is to determine which of the vertices has the smaller
active closed star in $\hat{\complex}_i$, or equivalently,
which vertex has the smaller closed star in $\complex_i$.
The size of the closed star of a vertex $v$ could be computed
by a simple graph traversal in $\dict$, starting at a vertex $v$
and following the cofacet pointers recursively, counting the number
of simplices encountered. However, we want to identify
the smaller star with only $O(c_i)$ operations, and the closed star
can be much larger. Therefore, we change the traversal in several ways:

First of all, observe that $|\cstar(u)|\leq |\cstar(v)|$ if and only if $|\ostar(u)|\leq |\ostar(v)|$ (in $\complex_i$).
Now define $\ostar(u,\neg v):=\ostar(u)\setminus\ostar(v)$. 
Then, $|\ostar(u)|\leq |\ostar(v)|$ if and only if $|\ostar(u,\neg v)|\leq |\ostar(v,\neg u)|$,
because we subtracted the intersection of the stars on both sides.
Finally, note that $\min \{|\ostar(u,\neg v)|,|\ostar(v,\neg u)|\} \leq c_i$,
as one can easily verify.
Moreover, we can count the size of $\ostar(u,\neg v)$ by a cofacet traversal from $u$,
ignoring cofacets that contain $v$ (using the key of $\innerdict_\ast$),
in $O(|\ostar(u,\neg v)|)$ time. However, this is still not enough, because
counting both sets independently gives a running time of $\max \{|\ostar(u,\neg v)|,|\ostar(v,\neg u)|\}$.
The last trick is that we count the sizes of $\ostar(u,\neg v)$ and $\ostar(v,\neg u)$ at the same time
by a simultaneous graph traversal of both, terminating as soon as one of the traversal stops.
The running time is then proportional to $2 \cdot \min\{|\ostar(u,\neg v)|,|\ostar(v,\neg u)|\} = O(c_i)$,
as required.

Assume w.l.o.g. that $|\cstar(u)|\leq |\cstar(v)|$. Also in time $O(c_i)$, we can obtain $\ostar(u,\neg v)$.
We sort its elements by increasing dimension, which can be done in $O(c_i+\towerdim)$ using integer sort.
For each simplex $\sigma = \{u,v_1,\ldots,v_k\} \in \ostar(u,\neg v)$ in order, we check whether $\{v,v_1,\ldots,v_k\}$
is in $\dict$. If not, we add it to $\dict$ and also write it to the output stream.
Then, we write $\{u,v,v_1,\ldots,v_k\}$ to the output stream (note that we do not have to check it existence in $\hat{\complex}_i$, 
because it does not by construction, and there is no need to add it to $\dict$ because of the next step).
At the end of the loop,
we wrote exactly the simplices in $\complex_{i+1}\setminus\complex_i$ to the output stream,
which proves correctness.

It remains to maintain the invariant on $\dict$. Assuming still that $|\cstar(u)|\leq |\cstar(v)|$,
$u$ turns inactive in $\hat{\complex}_{i+1}$. We simply traverse over all cofaces of $u$
and remove all encountered simplices from $\dict$. After this operations, the invariant holds.
This ends the description of the algorithm.

\paragraph{Complexity analysis.}
By applying the operation costs on the above described algorithm, we obtain the following statement.
Combined with Propositions~\ref{prop:barcode_equiv} and~\ref{prop12}, it completes the proof of Theorem~\ref{thm:main_theorem_1}.

\begin{proposition} \label{prop:complexity}
	The algorithm requires $O(\towerdim \cdot \towerwidth)$ space and $O(\towerdim \cdot |\hat{\complex}_m| \cdot C_{\towerwidth})$ time, where
	$\towerwidth = \max_{i=0,\ldots,m}|\complex_i|$ and $C_{\towerwidth}$ is the cost of an operation in a dictionary with at most $\towerwidth$ 
	elements.
\end{proposition}

\begin{proof}
	The space complexity follows at once from the invariant, because the size of $\dict$ is at most $O(\towerdim|\complex_i|)$
	during the $i$-th iteration. 

	For the time bound, we set $S:=|\hat{\complex}_m|$ for convenience and 
	show that the algorithm finishes in $O(\towerdim\cdot S)$ steps,
	assuming dictionary operations to be of constant cost.
	We have one simplex insertion per elementary inclusion which requires $O(\towerdim)$ operations. 
	Thus, all inclusions are bounded by $O(n\towerdim)$, which is subsumed by our bound as $n\leq S$.
	For the contraction case, we need $O(c_i)$ to identify the smaller star, $O(c_i+\towerdim)$
	to get a sorted list of $\ostar(u,\neg v)$ (or vice versa), and $O(\towerdim\cdot c_i)$
	to add new vertices to $\dict$. Moreover, we delete the star of $u$ from $\dict$;
	the cost for that is $O(\towerdim \cdot d_i)$, where $d_i$ is the number of deleted simplices.
	Thus, the complexity of a contraction is $O(\towerdim \cdot (c_i + d_i))$.

	Since $c_i$ is the number of simplices added to the filtration at step $i$,
	the sum of all $c_i$ is bounded by $O(S)$.
	Moreover, because every simplex that ever appears in $\dict$ belongs to $\hat{\complex}_m$ 
	and every simplex is inserted only once,
	the sum of all $d_i$ is bounded by $O(S)$ as well.
	Therefore, the combined cost over all contractions is $O(\towerdim\cdot S)$ as required. 
\end{proof}

Note that the dictionary $\dict$ has lists of identifiers of length up to $\towerdim+1$ as keys, so that
comparing two keys has cost $O(\towerdim)$. Therefore, using balanced binary trees as dictionary,
we get a complexity of $C_{\towerwidth}=O(\towerdim \log\towerwidth)$.
Using hash tables, we get an expected worst-case complexity of $C_{\towerwidth}=O(\towerdim)$.

\paragraph{Experimental results.}
The following tests where made on a 64-bit Linux (Ubuntu) HP machine with a 3.50 GHz Intel processor and 63 GB RAM. The programs were all 
implemented in C++ and compiled with optimization level -O2. Our algorithm was implemented in the software Sophia.\footnote{\url{https://bitbucket.org/schreiberh/sophia/}}

To test its performance, we compared it to the software 
Simpers\footnote{\url{http://web.cse.ohio-state.edu/~tamaldey/SimPers/Simpers.html}} (downloaded in August 2017),
which is the implementation of the Annotation Algorithm from Dey, Fan and Wang described in \cite{dfw-computing}. 
Simpers computes the persistence of the given filtration, so we add to our time the time the library PHAT (version 1.5) needs to compute the
persistence from the generated filtration. PHAT was used with its default parameters and its '-{}-ascii -{}-verbose' options activated. 
Simpers also used its default parameters except for the dimension parameter which was set to 5.

The results of the tests are in Table~\ref{tab:resultsAlg1}. 
The timings for File IO are not included in any process time except the input reading of Sophia.
The memory peak was obtained via the '/usr/bin/time -v' Linux command. Each command was repeated 10 times and the average was token.
The first three towers in the table, data1-3, were generated incrementally on a set of $n_0$ vertices: 
In each iteration, with $90 \%$ probability, a new simplex is included, that is picked uniformly at random among the simplices 
whose facets are all present in the complex, and with $10 \%$ probability, two randomly chosen vertices of the complex are contracted. 
This is repeated until the complex on the remaining $k$ vertices forms a $k-1$-simplex, in which case no further simplex can be added.
The remaining data was generated from the SimBa (downloaded in June 2016) library with default parameters using the point clouds from
\cite{dsw-simba}. To obtain the towers that SimBa computes internally, we included a print command at a suitable location in the SimBa code.

\begin{table}[h]
	\centering
	\small
	\textsf{
	\makebox[\textwidth][c]{
	\begin{tabular}{|l|*{10}{r|}}
\cline{7-11}
\multicolumn{6}{l|}{}
	& \multicolumn{3}{c|}{\texttt{Sophia + PHAT}}	& \multicolumn{2}{c|}{\texttt{Simpers}} \\
\cline{2-11}
\multicolumn{1}{l|}{}
	& \makecell*[c]{$c$}	& \makecell*[c]{$n$}	& \makecell*[c]{$n_0$}	& \makecell*[c]{$\towerdim$}	& \makecell*[c]{$\towerwidth$}
	& \makecell*[c]{filtration \\ size}	& \makecell*[c]{time \\ (s)}	& \makecell*[c]{mem. peak (kB) \\ \texttt{Sophia} / total}
	& \makecell*[c]{time \\ (s)}	& \makecell*[c]{mem. peak \\ (kB)} \\
\hline
\texttt{data1}
	& 495			& 4\,833		& 500			& 4				& 2\,908
	& 19\,747				& 0.07				& 4\,752 / 6\,472
	& 0.54				& 10\,030 \\
\texttt{data2}
	& 795			& 7\,978		& 800			& 4				& 4\,816
	& 35\,253				& 0.20				& 5\,286 / 9\,259
	& 2.82				& 19\,876 \\
\texttt{data3}
	& 794			& 8\,443		& 800			& 5				& 5\,155
	& 38\,101				& 0.21				& 5\,638 / 9\,487
	& 3.88				& 25\,104 \\
\texttt{GPS}
	& 1\,746		& 8\,585		& 1\,747		& 3				& 1\,747
	& 9\,063				& 0.02				& 4\,234 / 5\,027
	& 0.07				& 5\,849 \\
\texttt{KB}
	& 22\,499		& 95\,019		& 22\,500		& 3				& 22\,500
	& 133\,433				& 0.30				& 10\,036 / 14\,484
	& 0.51				& 25\,392 \\
\texttt{MC}
	& 23\,074		& 143\,928		& 23\,075		& 3				& 28\,219
	& 185\,447				& 0.51				& 13\,730 / 18\,792
	& 0.77				& 26\,718 \\
\texttt{S3}
	& 252\,995		& 1\,473\,580		& 252\,996		& 4				& 252\,996
	& 1\,824\,461				& 8.50				& 85\,128 / 151\,860
	& 10.78				& 247\,956 \\
\texttt{PC25}
	& 14\,999		& 10\,246\,125		& 15\,000		& 3				& 2\,191\,701
	& 12\,283\,003				& 135.02			& 994\,400 / 1\,439\,664
	& $\infty$			&- \\
\hline
	\end{tabular}
	}
	}
	\caption{Experimental results. The symbol $\infty$ means that the calculation time exceeded 12 hours.}\label{tab:resultsAlg1}
\end{table}

To verify that the space consumption of our algorithm
does not dependent on the length of the tower, 
we constructed an additional example whose size exceeds our RAM capacity, but 
whose width is small:
we took 10 random points moving on a flat torus 
in a randomly chosen fixed direction.
When two points get in distance less than $t_1$ to each other,
we add the edge between them (the edge remains also if the points
increase their distance later on). When two points get in distance less than
$t_2$ from each other with $t_2 < t_1$, we contract the two vertices
and let a new moving point appear somewhere on the torus.
This process yields a sequence of graphs, and we take its flag complex
as our simplicial tower. In this way, we obtain a tower
of length about $3.5\cdot 10^9$ which has a file size of about $73$ GB,
but only has a width of $367$.
Our algorithm took about $2$ hours to convert this tower into a filtration
of size roughly $4.6\cdot 10^9$. During the conversion,
the virtual memory used was constantly around $22$ MB and the resident set 
size about $3.8$ MB only, confirming the theoretical prediction that
the space consumption is independent of the length of the tower.
The information about the memory use was taken from the 
'/proc/\textless pid\textgreater/stat' system file every 100\,000 insertions/contractions during the process.

\subsection{Tightness and lower bounds}\label{ssec:tightness_bounds}

The conversion theorem (Theorem~\ref{thm:main_theorem_1}) yields
an upper bound of $O(\towerdim \cdot n\log n_0)$ for the size of a filtration
equivalent to a given tower. It is natural to ask whether this bound can
be improved. In this section, we assume for simplicity that the maximal
dimension $\towerdim$ is a constant. In that case, it is not difficult
to show that our size analysis cannot be improved:

\begin{proposition} \label{prop:tightness}
	There exist an example of a tower with $n$ simplices and $n_0$ vertices
	such that our construction yields a filtration
	of size $\Omega(n\log n_0)$.
\end{proposition}

\begin{proof}
	Let $p = 2^k$ for some $k\in\N$. Consider a graph with $p$ edges
	$(a_i,b_i)$, with $a_1,\ldots,a_p$, $b_1,\ldots,b_p$ $2p$ distinct vertices.
	Our tower first constructs this graph with inclusions (in an arbitrary order).
	Then, the $a$-vertices are contracted in a way such that the contracting forest
	is a fully balanced binary tree~-- see Figure~\ref{fig:tightness_dim1}
	for an illustration.
	
	\begin{figure}[h]
		\centering

		\includegraphics[width=100mm]{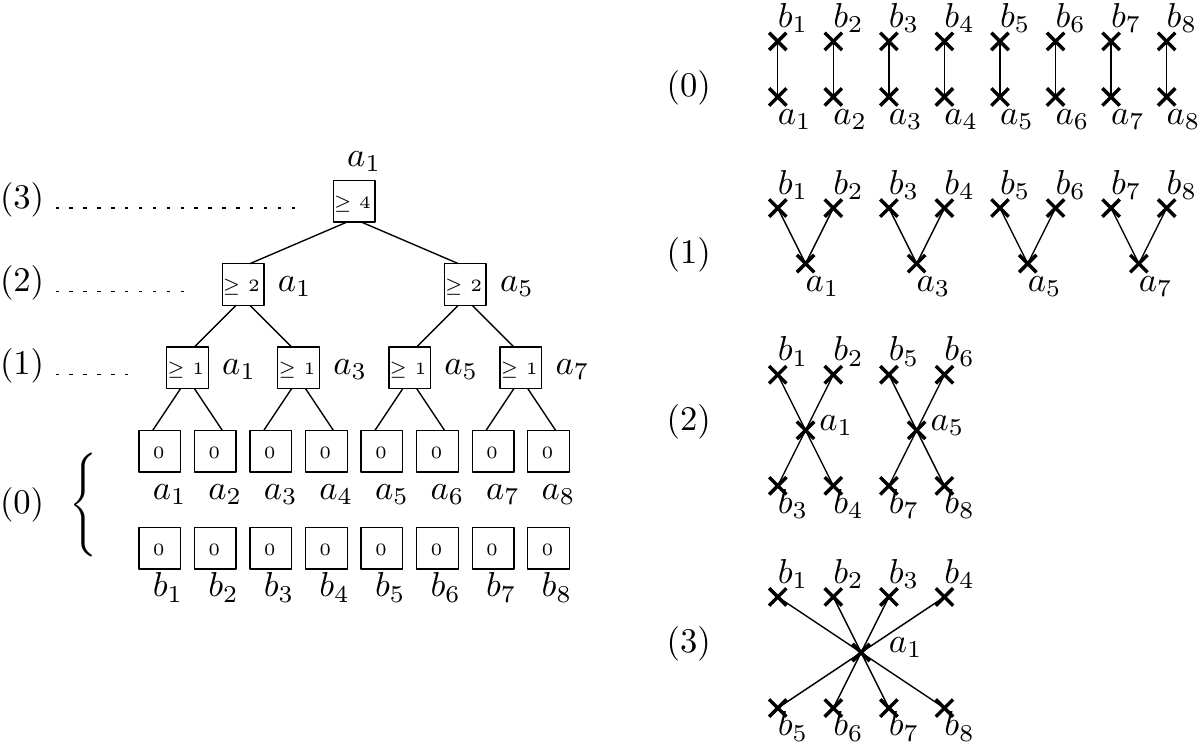}

		\caption{[In proof of Proposition~\ref{prop:tightness}] Sequence of contractions in the described construction for $p = 8$ (right)
		and the corresponding contracting forest (left), whose nodes contains the cost of the contractions} 
		\label{fig:tightness_dim1}
	\end{figure}

	To bound the costs, it suffices to count the number of edges added between 
	$a$-vertices and $b$-vertices in each step. We call them $ab$-edges
	from now on. Define the \emph{level}
	of a contraction to be its level in the contracting tree, where $0$ is the
	level of the leaves, and $k$ is the level of the root.
	On level $1$, a contraction of $a_i$ and $a_j$ yields exactly one new $ab$-edge, either $(a_i,b_j)$ or $(a_j,b_i)$. 
	The resulting contracted vertex has two incident $ab$-edges. A contraction on level $2$ 
	yields two novel $ab$-edges, and a vertex with $4$ incident $ab$-edges.
	By a simple induction, we observe that a contraction on level $i$
	introduces $2^{i-1}$ new $ab$-edges, 
	and hence has a cost of at least $2^{i-1}$, for $i=1,\ldots,k$.
	This means that the sum of the costs of all level $i$ contractions
	is exactly $\frac{p}{2}$. Summing up over all $i$ yields a cost of at least
	$k\frac{p}{2}$. The result follows because $k = \log (n_0/2)$ and $p = n/3$.
\end{proof}

Another question is whether a different approach could convert a tower
into a filtration with an (asymptotically) smaller number of simplices.
For this question, consider the \emph{inverse persistence computation problem}:
given a barcode, find a filtration of minimal size which realizes this barcode.
Note that solving this problem results in a simple solution for the conversion
problem: compute the barcode of the tower first using an arbitrary algorithm;
then compute a filtration realizing this barcode.
While useful for lower bound constructions, 
we emphasize the impracticality of this solution, as the main purpose
of the conversion is a faster computation of the barcode.

Let $b$ be the number of bars of a barcode. 
An obvious lower bound for the filtration size is $\Omega(b)$,
because adding a simplex causes either the birth or the death 
of exactly one bar in the barcode. 
For constant dimension, $O(b)$ is also an upper bound:

\begin{lemma} \label{lem:constr_from_barcode}
	For a barcode with $b$ bars and maximal dimension $\Delta$,
	there exists a filtration of size $\leq 2^{\Delta+2} b$ realizing this barcode.
\end{lemma}

\begin{proof}
	Begin with an empty complex. Now consider the birth and death times represented by the barcode one by one. 
	The first birth will be the one of a $0$-dimensional class, 
	so add a vertex $v_0$ to the complex. 
	From now, every time a $0$-dimensional homology class is born add a new vertex to the complex. When a $0$-dimensional homology class dies, 
	link the corresponding vertex to $v_0$ with an edge.

	When a $k$-dimensional homology class is born, with $k > 0$, add the boundary of a $(k+1)$-simplex to the complex 
	that is incident to $v_0$ and to $k$ novel vertices. When this homology class dies, 
	add the corresponding $(k+1)$-simplex.
	This way, the resulting filtration realizes the barcode. 
	For a bar of the barcode in dimension $k$, we have to add all proper faces of a $(k+1)$-simplex (except for one vertex). 
	Since that number is at most $2^{k+2}-3$, the result follows.
\end{proof}

If a tower has length $m$, what is the maximal 
size of its barcode?
If the size of the barcode is $O(m)$, the preceding lemma
implies that a conversion to a filtration of linear size is possible
(for constant dimension).
On the other hand, any example of a tower yielding a super-linear lower
bound would imply a lower bound for any conversion algorithm, because
a filtration has to contain at least one simplex per bar in its barcode.

To approach the question, observe that
a single contraction might destroy many homology classes at once:
consider a ``fan'' of $t$ empty triangles, all glued together along a common
edge $ab$ (see Figure~\ref{fig:contr_and_homology} (a)). 
Clearly, the complex has $t$ generators in $1$-homology. 
When $a$ and $b$ are contracted, the complex transforms to a star-shaped graph
which is contractible.
Moreover, a contraction can also create many homology classes at once:
consider a collection of $t$ disjoint triangulated spheres, all glued together along
a common edge $ab$. For every sphere, remove one of the triangles incident
to $ab$ (see Figure~\ref{fig:contr_and_homology} (b)). 
The resulting complex is acyclic. The contraction of the edge $ab$
``closes'' each of the spheres, and the resulting complex has $t$ generators
in $2$-homology.
Finally, a contraction might not not affect the homology at all~--
see Figure~\ref{fig:contr_and_homology} (c).

\begin{figure}[h]
	\centering

	\includegraphics[width=84mm]{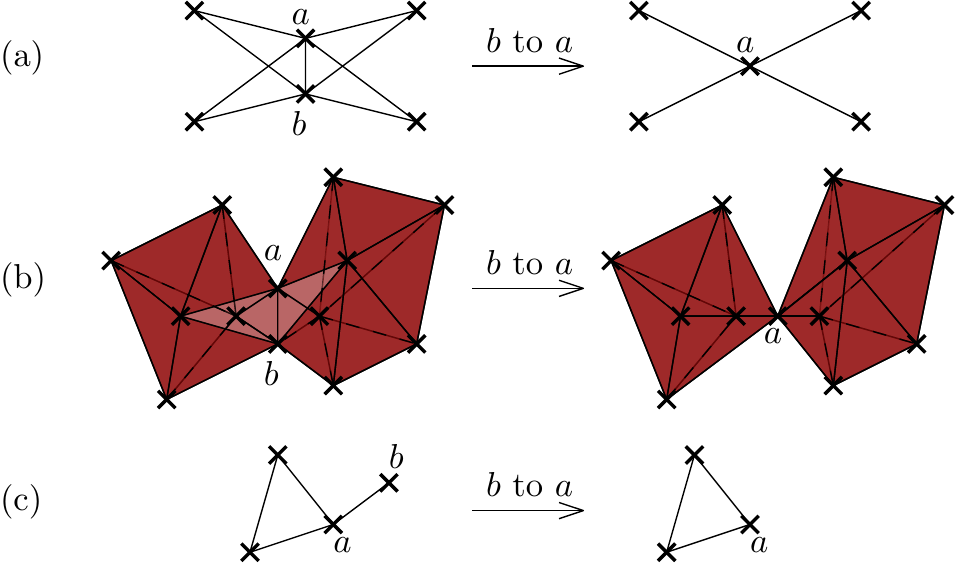}

	\caption{Examples of the influence of contractions on the homology classes: 
	(a) destruction of four $1$-homology generators, (b) creation of two $2$-homology generators, and (c) no influence at all} 
	\label{fig:contr_and_homology}
\end{figure}

The above examples show that a single contraction can create and destroy 
many bars.
For a super-linear bound on the barcode size, however, we would 
have to construct an example where sufficiently many contractions create
a large number of bars. 
So far, we did neither succeed in constructing such an example,
nor are we able to show that such an example does not exist.

\section{Persistence by Streaming}\label{sec:persistence_by_streaming}

Even if the complex we need to maintain in memory during the conversion is relatively small, at the end, the algorithm to compute the final 
persistence still needs to memorize the whole filtration we give it as input. If the complex in the original tower has an interesting maximum size 
during the whole process, we should be able to compute its persistence even if the tower is extremely long. So we design here a streaming variation 
of the reduction algorithm that computes the barcode of filtrations, such that it has an efficient memory use. More precisely, we will prove 
the following theorem:

\begin{theorem} \label{thm:main_theorem_2}
	With the same notation as in Theorem~\ref{thm:main_theorem_2}, 
	we can compute the barcode of a tower $\tower$ in worst-case time
	$O(\towerwidth^2\cdot\towerdim\cdot n\cdot\log n_0)$
	and space complexity $O(\towerwidth^2)$
\end{theorem}

We describe the algorithm in Section~\ref{ssec:reduction_algorithm} and prove the complexity bounds
in Section~\ref{ssec:reduction_complexity}. The described algorithm requires various adaptations
to become efficient in practice, and we describe an improved variant in Section~\ref{ssec:reduction_implementation}.
We present some experiments in Section~\ref{ssec:reduction_experiments}.

\subsection{Algorithmic description} \label{ssec:reduction_algorithm}

On a high level, our algorithm converts the tower into an equivalent filtration 
and computes the barcode of that filtration. 
We focus on the second part, which we describe as a streaming algorithm.
The input to the algorithm is a stream of elements, each starting with a token \{\texttt{ADDITION}, \texttt{INACTIVE}\}
followed by a simplex identifier which represents a simplex $\sigma$.
In the addition case, this is followed by a list of simplex identifiers specifying the facets of $\sigma$.
In other words, the element encodes the next column of the boundary matrix.
For the inactive case, it means that $\sigma$ has become inactive in the complex. 
In particular, no subsequent simplex in the stream will contain $\sigma$ as a facet.
It is not difficult to modify the algorithm from Section~\ref{ssec:algorithm_for_conversion}
to return a stream as required, within the same complexity bounds.

The algorithm uses a matrix data type $\dsmatrix$ as its main data structure.
We realize $\dsmatrix$ as a dictionary of columns, indexed
by a simplex identifier. Each column is a sorted linked list of identifiers
corresponding to the non-zero row indices of the column.
In particular, we can access the pivot of the column in constant time
and we can add two columns in time proportional to the maximal size
of the involved lists. Note that most algorithms store the boundary matrix
as an array of columns, but we use dictionaries for space efficiency.

There are two secondary data structures that we mention briefly:
given a row index $r$, we have to identify the column index $c$
of the column that has $r$ as pivot in the matrix (or to find out
that no such column exists). This can be done using a dictionary
with key and value type both simplex identifiers. 
Finally, we maintain a dictionary
representing the set of simplex identifiers that represent active
simplices of the complex, plus a flag
denoting whether the corresponding simplex
is positive or negative. 
It is straight-forward to maintain these
structures during the algorithm, and we will omit the description
of the required operations.

The algorithm uses two subroutines. The first one, called \texttt{reduce\_column},
takes a column identifier $j$ as input
is defined as follows: iterate through the non-zero row indices of $j$.
If an index $i$ is the index of an inactive and negative column in $\dsmatrix$,
remove the entry from the column $j$ (this is the ``compression'' 
described at the end of Section~\ref{sec:background}). After this pre-processing, reduce the column:
while the column is non-empty, and its pivot $i$ is the pivot of another
column $k<j$ in the matrix, add column $k$ to column $j$.

The second subroutine, \texttt{remove\_row}, takes a index $\ell$ as input
and clears out all entries in row $\ell$ from the matrix.
For that, let $j$ be the column with $\ell$ as pivot.
Traverse all non-zero columns of the matrix except column $j$.
If a column $i \neq j$ has a non-zero entry at row $\ell$,
add column $j$ to column $i$.
After traversing all columns, remove column $j$ from $\dsmatrix$.

The main algorithm can be described easily now: if the input stream
contains an addition of a simplex, we add the column to $\dsmatrix$
and call \texttt{reduce\_column} on it.
If at the end of that routine, the column is empty, it is removed from $\dsmatrix$.
If the column is not empty and has pivot $\ell$,
we report $(\ell,j)$ as a persistence pair and
check whether $\ell$ is active. 
If not, we call \texttt{remove\_row($\ell$)}.
If the input stream specifies that simplex $\ell$ becomes inactive,
we check whether $j$ appears as pivot in the matrix and call 
\texttt{remove\_row($\ell$)} in this case.

\begin{proposition} \label{prop:alg_correctness}
	The algorithm computes the correct barcode.
\end{proposition}

\begin{proof}
	First, note that removing a column from $\dsmatrix$ within the 
	procedure \texttt{remove\_row} does not affect further reduction steps:
	Since the pivot $\ell$ of the column is inactive, no subsequent column
	in the stream will have an entry in row $\ell$. Moreover, the reduction
	process cannot introduce an entry in row $\ell$ because the routine has removed
	all such entries.

        Note that \texttt{remove\_row} might also include right-to-left column additions,
        and we also have to argue that they do not change the pivots.
        Let $\rmatrix$ denote the matrix obtained by the standard compression algorithm
        that does not call \texttt{remove\_row} (as described in Section~\ref{sec:background}).
        Just before our algorithm calls \texttt{reduce\_column}
        on the $j$-th column, let $\otherrmatrix$ denote the matrix with $(j-1)$ columns
        that is represented by $\dsmatrix$. It is straight-forward to verify by an inductive argument
        that every column of $\otherrmatrix$ is a linear combination of $\rmatrix_1,\ldots,\rmatrix_{j-1}$.
        \texttt{reduce\_column} adds a subset of the columns of $\otherrmatrix$ to the $j$-th column.
        Thus, the reduced column can be expressed by a sequence of left-to-right column additions
        in $\rmatrix$, and thus yields the same pivot as the standard compression algorithm.
\end{proof}

\subsection{Complexity analysis}\label{ssec:reduction_complexity}

We analyze how large the structure $\dsmatrix$ can become during the algorithm.
After every iteration, the matrix represents the reduced boundary matrix of some
intermediate complex $\hat{\othercomplex}$ with
$\hat{\complex}_i\subseteq\hat{\othercomplex}\subseteq\hat{\complex}_{i+1}$
for some $i=0,\ldots,m$. Moreover, the active simplices define
a subcomplex $\othercomplex\subseteq\hat{\othercomplex}$
and there is a moment during the algorithm where $\hat{\othercomplex}=\hat{\complex}_i$
and $\othercomplex=\complex_i$, for every $i=0,\ldots,m$. 
We call this the \emph{$i$-th checkpoint}.
We will make frequent use of the following simple observation.

\begin{lemma}\label{lem:intermediate_size}
	$|\hat{\complex}_{i+1} \setminus \hat{\complex}_i| \leq |\complex_i| \leq \towerwidth$
\end{lemma}

\begin{lemma} \label{lem:matrix_width}
	At every moment, the number of columns stored in $\dsmatrix$ is at most $2\towerwidth$.
\end{lemma}

\begin{proof}
	It can be verified easily that throughout the algorithm,
	a column is stored in $\dsmatrix$ only if not zero, and its pivot is active.
	So, assume first that we are at the $i$-th checkpoint for some $i$.
	Since $\othercomplex=\complex_i$, there are not more than $|\complex_i|\leq\towerwidth$
	active simplices. Since each column has a different active pivot, their number is
	also bounded by $\towerwidth$.
	If we are between checkpoint $i$ and $i+1$, there have been not more than
	$\towerwidth$ columns added to $\dsmatrix$ since the $i$-th checkpoint from Lemma~\ref{lem:intermediate_size}.
	The bound follows.
\end{proof}

The number of rows is more difficult to bound because we cannot guarantee that each column in $\dsmatrix$
corresponds to an active simplex. Still, the number of rows is asymptotically the same as for columns:

\begin{lemma} \label{lem:matrix_height}
	At every moment, the number of rows stored in $\dsmatrix$ is at most $4\towerwidth$.
\end{lemma}

\begin{proof}
	Consider a row index $\ell$ and a time in the algorithm where $\dsmatrix$ represents $\hat{\othercomplex}$. 
	By the same argument as in the previous lemma, we can argue that there are at most $2 \towerwidth$ active
	row indices at any time. 
        Therefore, we restrict our attention to the case that $\ell$ is inactive
	and distinguish three cases. If $\ell$ represents a negative simplex, 
        we observe that its row should have been was removed
        due to the compression optimization,
	and after $\ell$ became inactive, no simplex can have it as a facet either.
	It follows that row $\ell$ is empty in this case.
	If $\ell$ is positive and was paired with another index $j$ during the algorithm,
	then \texttt{remove\_row} was called on $\ell$, either at the moment the pair
	was formed, or when $\ell$ became inactive. Since the procedure removes the row,
	we can conclude that row $\ell$ is empty also in this case.
	The final case is that $\ell$ is positive, but has not been paired so far.
	It is well-known that in this case, $\ell$ is the generator of an homology class
	of $\hat{\othercomplex}$. Let 
	\[
		\beta(\hat{\othercomplex}) := \sum_{i=0}^{\towerdim}\beta_i(\hat{\othercomplex})
	\]
	denote the sum of the Betti numbers of the complex. Then, 
	it follows that the number of such row indices is at most $\beta(\hat{\othercomplex})$.
	
	We argue that $\beta(\hat{\othercomplex})\leq 2\towerwidth$ which proves our claim. 
	Assume that $\hat{\complex}_i\subseteq\hat{\othercomplex}\subset\hat{\complex}_{i+1}$.
	We have that 
	$\beta(\hat{\complex}_i)=\beta(\complex_i)$ by Lemma~\ref{lem3},
	and since $\complex_i$ has at most $\towerwidth$ simplices,
	$\beta(\complex_i)\leq\towerwidth$. Since we add at most $\towerwidth$ simplices
	to get from $\hat{\complex}_i$ to $\hat{\othercomplex}$, and each addition can 
	increase $\beta$ by at most one, we have indeed that 
	$\beta(\hat{\othercomplex})\leq 2\towerwidth$.
\end{proof}

\begin{proposition} \label{prop:pers_alg_complexity}
	The algorithm has time complexity $O( \towerwidth^2\cdot\towerdim\cdot n\cdot\log n_0)$
	and space complexity $O(\towerwidth^2)$.
\end{proposition}

\begin{proof}
	The space complexity is immediately clear from the preceding two lemmas,
	as $\dsmatrix$ is the dominant data structure in terms of space consumption.
	For the time complexity, we observe that both subroutines \texttt{reduce\_column}
	and \texttt{remove\_row} need $O(\towerwidth)$ column additions and $O(\towerwidth)$
        dictionary operations in the worst case. A column addition costs $O(\towerwidth)$,
        and a dictionary operation is not more expensive (since the dictionaries contain
        at most $O(\towerwidth)$ elements and their keys are integers). 
        So, the complexity of both methods is $O(\towerwidth^2)$.
        Since each routine is called at most once per input element, and there are
	$O(\towerdim\cdot n\cdot\log n_0)$ elements by Theorem~\ref{thm:main_theorem_1},
	the bound follows.
\end{proof}

\subsection{Implementation} \label{ssec:reduction_implementation}

The algorithm described in Section~\ref{ssec:reduction_algorithm} is not efficient
in practice for three reasons: First, it has been observed as a general rule
that the \emph{clearing optimization}~\cite{bkrw-phat} significantly
improves the standard algorithm. However, in the above form, that optimization
is not usable because it requires to process the columns in a non-incremental way.
Second, the \texttt{remove\_row} routine scans the entire matrix $\dsmatrix$;
while not affecting the worst-case complexity, frequently scanning the matrix 
should be avoided in practice.
Finally, the above algorithm uses lists to represent columns,
but it has been observed that this is a rather inefficient way to perform
matrix operations~\cite{bkrw-phat}. 

We outline a variant of the above algorithm that partially overcomes these drawbacks
and behaves better in practice. In particular, our variant can be implemented
with all column representations available in the \textsc{Phat} library.
The idea is to perform a ``batch'' variant of the previous algorithm:
We define a \emph{chunk size} $\chunk$ and read in $\chunk$ elements
from the stream; we insert added columns in the matrix,
removing row entries of already known inactive negative columns as before,
but not reducing the columns yet.
After having read $\chunk$ elements, we start the reduction of the newly inserted columns
using the clearing optimization. That is, we go in decreasing dimension
and remove a column as soon as its index becomes the pivot of another column;
see~\cite{ck-twist} for details.
After the reduction ends, except for the last chunk, we go over the columns of the matrix and check for each
pivot whether it is active. If it is, we traverse its row entries in decreasing order, but skipping the pivot. 
Let $\ell$ be the current entry. If $\ell$ is the inactive pivot of some column $j$, we add $j$ to the current column.
If $\ell$ is inactive and represents a negative column, we delete $\ell$ from the current column.
After performing these steps for all remaining columns of the matrix, we go over all columns again, deleting every column with inactive pivot.
(By the way, also cleaning up the secondary data structures described in \ref{ssec:reduction_algorithm}.)

\smallskip

It remains the question of how to choose the parameter $\chunk$. The chunk provides
a trade-off between time and space efficiency. Roughly speaking, the matrix can have
up to $O(\towerwidth+\chunk)$ columns during this reduction, but the larger the chunks
are, the more benefit one can draw from the clearing optimization (the clearing
optimization fails for pairs where the simplices are in different chunks).
We therefore recommend to choose $\chunk$ rather large, but making sure that the 
matrix will still fit into memory.

\subsection{Experimental evaluation}\label{ssec:reduction_experiments}

The tests were made with the same setup as in Section \ref{ssec:algorithm_for_conversion}. 
Figure~\ref{fig:chunk_dgm} shows the effect of the chunk size
parameter $C$ on the runtime and memory consumption of the algorithm.
The data used is \texttt{S3} (see Section \ref{ssec:algorithm_for_conversion});
we also performed the tests on the other examples 
from Table~\ref{tab:resultsAlg1}, with similar outcome.
The File IO operations are included in the measurements. Confirming the theory, as the chunk size decreases, 
our implementation needs less space 
but more computation time (while the running time
seems to increase slightly again for larger chunk sizes).

\begin{figure}[h]
	\centering

	\includegraphics[width=129mm]{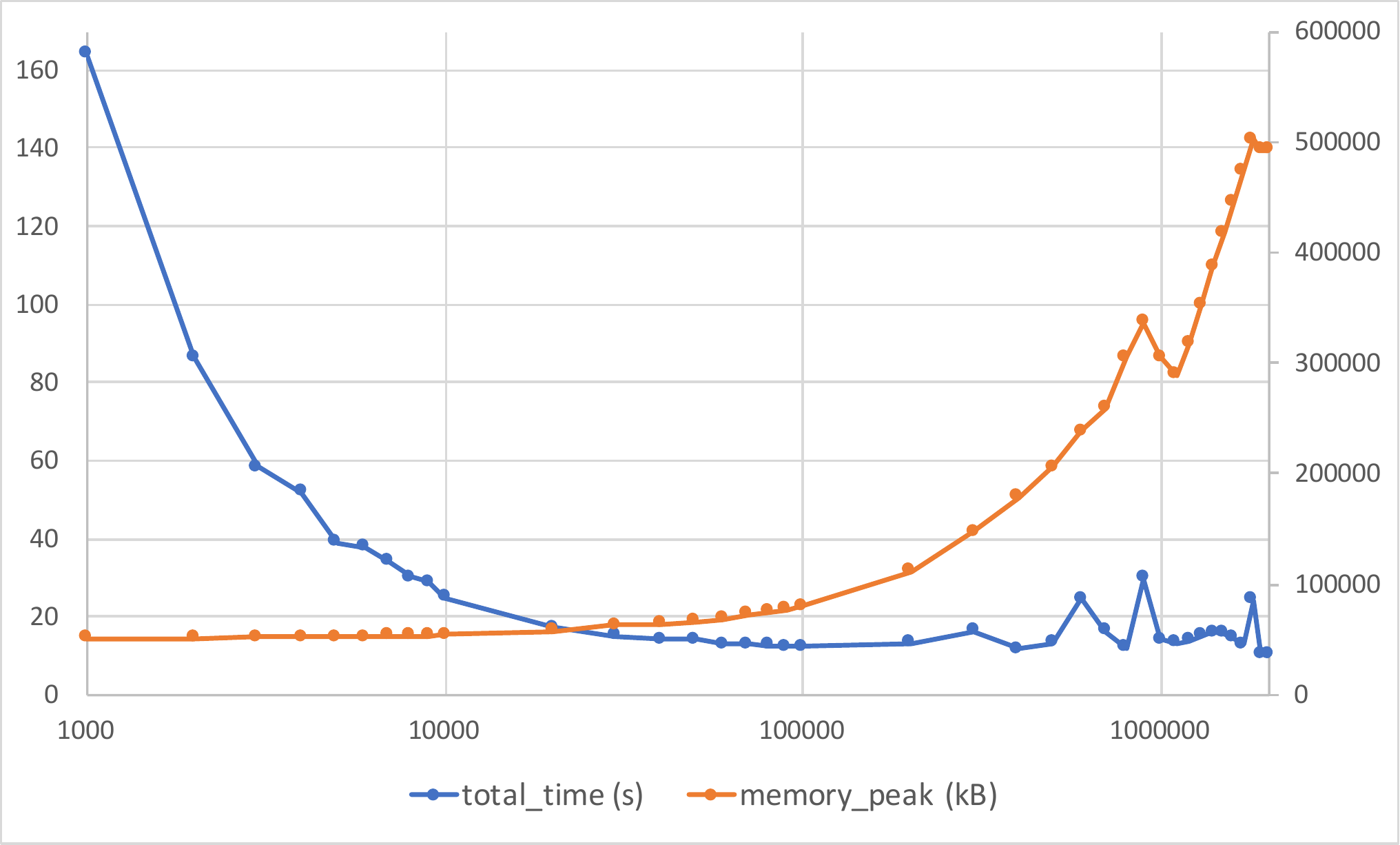}

	\caption{Evolution of processing time (left Y-axis in sec) and process memory peak (right Y-axis in kB) 
	depending on the chunk size (logarithmic X-axis)} 
	\label{fig:chunk_dgm}
\end{figure}

For the $4.6\cdot 10^9$ inclusions tower from Section \ref{ssec:algorithm_for_conversion}, with $\chunk = 200\,000$, 
the algorithm took around 4.5 hours, 
the virtual memory used was constantly around $68$ MB and 
the resident set size constantly around $49$ MB, confirming the theoretical
statement that the memory size does not depend on the length of the filtration.

\section{Conclusion}
In the first part of the paper, we have presented an efficient algorithm
to reduce the computation of the barcode of a simplicial tower to the computation
of a barcode of a filtration with slightly larger size. With our approach,
every algorithmic improvement for persistence computation on the 
filtration case becomes immediately applicable to the case of towers as well.
In the second part of the paper, we present a streaming variant of the classical
persistence algorithm for the case of towers. In here, 
the extra information provided by the towers allows a more space efficient
storage of the boundary matrix.

There are various theoretical and practical questions remaining for further
work: as already exposed in 
Section~\ref{ssec:tightness_bounds}, the question of how large can the barcode
of a tower become has immediate consequences on the conversion from towers
to filtrations. Moreover, while we focused on constant dimension in 
Section~\ref{ssec:tightness_bounds}, we cannot exclude the possibility
that our algorithm achieves a better asymptotic bound for non-constant dimensions.

We made our software publicly available in the Sophia library.
There are several open questions regarding practical performance.
For instance, our complex representation, based on hash table,
could be replaced with other variants, such as the Simplex tree~\cite{simplextree}. 
Moreover, it would be interesting to compare our streaming approach
for the barcode computation with a version that converts to a filtration
and subsequently computes the barcode with the annotation algorithm.
The reason is that the latter algorithm only maintains a cohomology basis
of the currently active complex in memory and therefore avoids
the storage of the entire boundary matrix.

Since both the simplex tree and annotation algorithm are part of the Gudhi
library~\cite{gudhi:urm}, we plan to integrate our conversion algorithm
in an upcoming version of the Gudhi library, both to increase the usability
of our software and to facilitate the aforementioned comparisons.

\paragraph{Acknowledgements} The authors are supported by the Austrian Science Fund (FWF) grant number P 29984-N35.

\bibliographystyle{plain}
\bibliography{bib}

\end{document}